\documentclass[reqno]{amsart}

\usepackage{graphicx,subfigure,color}

\numberwithin{equation}{section}

\usepackage[latin1]{inputenc}
\usepackage[english]{babel}

\usepackage{amsmath,amsthm,amsfonts,latexsym,amssymb}
\usepackage[colorlinks]{hyperref}
\hypersetup{linkcolor=blue,citecolor=blue,filecolor=black,urlcolor=blue}
\usepackage{comment}

\usepackage{color}
\usepackage[dvipsnames]{xcolor}
\definecolor{darkgreen}{rgb}{0,0.7,0.1}

{ \theoremstyle{plain}
\newtheorem{theorem}{Theorem}[section]
\newtheorem{proposition}[theorem]{Proposition}
\newtheorem{lemma}[theorem]{Lemma}
\newtheorem{corollary}[theorem]{Corollary}
  \theoremstyle{remark}
\newtheorem{remark}[theorem]{Remark}
  \theoremstyle{definition}

}

\def\R{\mathbb{R}}
\def\Z{\mathbb{Z}}

\begin{document}


\title[]{A priori bounds and multiplicity of positive solutions for $p$-Laplacian Neumann problems with sub-critical growth}

\author[A. Boscaggin]{Alberto Boscaggin}
\address{Alberto Boscaggin\newline\indent 
Dipartimento di Matematica
\newline\indent
Università di Torino
\newline\indent
via Carlo Alberto 10, 10123 Torino, Italia}
\email{alberto.boscaggin@unito.it}

\author[F. Colasuonno]{Francesca Colasuonno}
\address{Francesca Colasuonno\newline\indent
Dipartimento di Matematica
\newline\indent
Alma Mater Studiorum Università di Bologna
\newline\indent
piazza di Porta S. Donato 5, 40126 Bologna, Italia}
\email{francesca.colasuonno@unibo.it}

\author[B. Noris]{Benedetta Noris}
\address{Benedetta Noris
\newline \indent Laboratoire Ami\'enois de Math\'ematique Fondamentale et Appliqu\'ee\newline\indent
Universit\'e de Picardie Jules Verne\newline\indent
33 rue Saint- Leu, 80039 AMIENS, France}
\email{benedetta.noris@u-picardie.fr}

\date{\today}

\begin{abstract} 
Let $1<p<+\infty$ and let $\Omega\subset\mathbb R^N$ be either a ball or an annulus.
We continue the analysis started in [Boscaggin, Colasuonno, Noris, ESAIM Control Optim. Calc. Var. (2017)], concerning quasilinear Neumann problems of the type
\[
-\Delta_p u = f(u), \quad u>0 \mbox{ in } \Omega, \quad
\partial_\nu u = 0 \mbox{ on } \partial\Omega.
\]
We suppose that $f(0)=f(1)=0$ and that $f$ is negative between the two zeros and positive after. In case $\Omega$ is a ball, we also require that $f$ grows less than the Sobolev-critical power at infinity. We prove a priori bounds of radial solutions, focusing in particular on solutions which start above 1. As an application, we use the shooting technique to get existence, multiplicity and oscillatory behavior (around 1) of non-constant radial solutions.  
\end{abstract}

\maketitle

\section{Introduction}
\subsection{Motivations, assumptions and main results}
In this paper we carry on the analysis started in \cite{ABF}, concerning quasilinear Neumann problems of the type
\begin{equation}\label{main}
\left\{
\begin{array}{ll}
\vspace{0.1cm}
-\Delta_p u = f(u) & \mbox{ in } \Omega \\
\vspace{0.1cm}
u > 0 & \mbox{ in } \Omega \\
\partial_\nu u = 0 & \mbox{ on } \partial\Omega, \\
\end{array}
\right.
\end{equation}
where $1<p<+\infty$, $\nu$ is the outer unit normal of $\partial \Omega$ and $\Omega\subset\mathbb R^N$ ($N\geq1$) is a radial domain which can be either an annulus 
$$
\Omega=\mathcal A(R_1,R_2):=\{ x \in \mathbb{R}^N \, : \, R_1 < \vert x \vert < R_2 \},\quad 0<R_1<R_2<+\infty,
$$
or a ball
$$
\Omega=\mathcal B(R_2):=\{ x \in \mathbb{R}^N \, : \vert x \vert < R_2 \},\quad 0<R_2<+\infty.
$$ 
Before stating precisely the hypotheses on $f$, we can have in mind, as a prototype nonlinearity, the following difference of pure powers 
\begin{equation}\label{eq:prototype}
f(s)=s^{q-1}-s^{r-1}\quad \mbox{with }\begin{cases}
p\le r<q\quad&\mbox{if }\Omega=\mathcal A(R_1,R_2),\\
p\le r<q<p^*&\mbox{if }\Omega=\mathcal B(R_2),
\end{cases}
\end{equation}
where, as usual,
\[
p^*:=\begin{cases}
\frac{Np}{N-p} \quad & \text{if } 1<p<N \\
+\infty & \text{if } p\geq N 
\end{cases}
\]
is the Sobolev critical exponent.

One of the main features of problem \eqref{main}, besides its radial symmetry, is that it admits a non-zero constant solution, say $u\equiv 1$, see condition $(f_\textrm{eq})$ below. 
We address existence of non-constant radial solutions $u\in W^{1,p}(\Omega)$ of \eqref{main}, as well as multiplicity, a priori bounds and oscillatory behavior around the constant solution. 
The recent literature has shown that, in presence of homogeneus Neumann boundary conditions, quasilinear equations of the type \eqref{main} typically admit many positive solutions (in addition to the constant one) and that the set of positive solutions has a rich structure. We quote here the articles 
\cite{LinNi88,LinNiTakagi88,budd1991asymptotic, AdimurthiYadava,AdimurthiYadava97,yadava1992note,WangWeiYan2010,SerraTilli2011, GrossiNoris,secchi2012increasing,BNW,BonheureGrossiNorisTerracini2015, bonheure2016multiple,ma2016bonheure,ColasuonnoNoris,BonheureCasterasNoris2017, CowanMoameni}, some of which will be discussed later.
Let us illustrate this fact in the semilinear case $p=2$, when $\Omega$ is a ball and $f(s)=s^{q-1}-s$ with $q>2$. In \cite{bonheure2016multiple} Bonheure, Grumiau and Troestler prove, via bifurcation analysis, the existence of multiple positive solutions, satisfying $u(0)<1$, and oscillating an increasing number of times around the constant 1. These solutions are a priori bounded independently of $q$, so that a certain type of solution (with a precise oscillatory behavior) which exists for a certain value of $q$ persists for larger values.
Under the additional assumption $q<2^*$, they further obtain solutions with $u(0)>1$, having similar properties. Some interesting numerical simulations (cf. \cite[Section 6, Fig. 16]{bonheure2016multiple}) suggest that the bifurcation branches of solutions with $u(0)>1$ and $q>2^*$ can have unpredictable behaviors, so that a type of solution which exists for a certain value of $q$ may not be present for subsequent values.

In \cite{ABF}, we investigate problem \eqref{main} in the general quasilinear case $1 < p < +\infty$, under a minimal set of assumptions for the nonlinear term $f$. More precisely, we show, via the shooting method, that existence, multiplicity and oscillatory behavior of radial solutions to \eqref{main} with $u(0)<1$ can still be provided, even with some remarkable novelties with respect to the semilinear case. We stress that no growth assumptions at infinity are required (just, $f(s) > 0$ for $s > 1$, see $(f_\textrm{eq})$ below).
Furthermore, in \cite[Section 6]{ABF}, we perform some numerical simulations which suggest that solutions with $u(0)>1$ do exist also in the quasilinear setting, for subcritical nonlinearities.

Motivated by the numerical evidence and by the analytical results for the semilinear case, in this paper we continue the description of \eqref{main}, by analyzing the existence and qualitative properties of solutions with $u(0)>1$ for general subcritical nonlinearities and for any $1<p<+\infty$. 
With respect to our previous paper \cite{ABF}, here we are facilitated by the fact of having a subcritical nonlinearity, which provides the needed compactness.
On the other hand, the main difficulty in the present paper is to obtain some a priori bounds on the solutions: while it is easy to show that solutions with $u(0)<1$ are a priori bounded (one can use energy methods as in \cite[Theorem 2.4]{bonheure2016multiple}), it costs us a big effort to obtain an analogous property for solutions with $u(0)>1$. 
Roughly speaking, we can say that a sequence of radial $W^{1,p}$-solutions in a ball, with zero radial derivative at the boundary, are allowed to explode only at the origin; the condition $u(0)<1$ automatically prevents this fact.

Let us now state the assumptions on $f$ which are required throughout the paper. As in \cite{ABF}, we assume:
\begin{itemize}
\item[$(f_\textrm{reg})$] $f \in \mathcal{C}([0,+\infty)) \cap \mathcal{C}^1((0,+\infty))$;
\item[$(f_\textrm{eq})$] $f(0) = f(1) = 0$, $f(s) < 0$ for $0 < s < 1$ and $f(s) > 0$ for $s > 1$; 
\item[$(f_0)$] $\liminf_{s\to0^+} \frac{f(s)}{s^{p-1}}  >-\infty$.
\end{itemize}
Furthermore, in case the domain is a ball, we impose in addition either
\begin{itemize}
\item[$(f_\mathrm{subl})$]  there exists $M \in (0,+\infty)$ such that $\limsup_{s\to+\infty} \frac{f(s)}{s^{p-1}}\leq M$;
\end{itemize}
or
\begin{itemize}
\item[$(f_\mathrm{subc})$] $\limsup_{s\to+\infty} \frac{f(s)}{s^{p-1}}=+\infty$ and $\exists \ \eta\in(0,1)$ s.t. $\limsup_{s\to+\infty} \frac{f^*(s)s}{F(\eta s)}<p^*$,
where $F(s):=\int_1^s f(\sigma)\,d\sigma$ for every $s>0$, and $f^*$ is the smallest non-decreasing function satisfying $f^*(s)\geq f(s)$ for every $s\geq1$, namely
$$
f^*(s):=\max_{\sigma\in [1,s]} f(\sigma)\quad\mbox{ for }s\geq1.
$$
\end{itemize}

\begin{remark}\label{rem:sottocriticita} 
(i) The two assumptions at infinity  $(f_\mathrm{subl})$ and $(f_\mathrm{subc})$, required when $\Omega$ is a ball, are complementary in the set of subcritical nonlinearities. They allow to consider both $(p-1)$-sublinear functions $f$, when $(f_\mathrm{subl})$ holds, and functions $f$ which are not $(p-1)$-sublinear, but have Sobolev-subcritical growth, when $(f_\mathrm{subc})$ is satisfied. Examples of functions satisfying $(f_\mathrm{subc})$ will be given below (see point (iii) of this Remark); on the other hand, we observe here that the subcritical growth of $f$ is a necessary condition for $(f_\mathrm{subc})$ to be fulfilled. More precisely, let $1<p<N$, then $(f_\mathrm{subc})$ implies that there exist $\varepsilon,\,C_\varepsilon>0$ and $s_\varepsilon>1$ such that
\begin{equation}\label{eq:subcrit}
f(s) \le  C_\varepsilon s^{p^*-1-\varepsilon} \quad \text{ for every } s>s_\varepsilon.
\end{equation}
In order to show it, it is enough to observe that, being $F$ monotone increasing and $f^*\geq f$ for $s\geq 1$, $(f_\mathrm{subc})$ provides $\limsup_{s\to+\infty} \frac{f(s)s}{F(s)}<p^*$.
Hence there exist $\varepsilon>0$ and $s_\varepsilon > 1$ such that
\[
\frac{f(s)}{F(s)}\leq \frac{p^*-\varepsilon}{s} \quad \text{ for every } s> s_\varepsilon.
\]
Integrating the previous inequality in $(s_\varepsilon,s)$, with $s>s_\varepsilon$, we deduce that there exists $C_\varepsilon>0$ such that \eqref{eq:subcrit} holds.

\medskip

\noindent (ii) The reason why the case of the annulus does not require any additional assumptions of the type $(f_{\mathrm{subl}})$ or $(f_\mathrm{subc})$ relies on the fact that for $R_1>0$ problem \eqref{main} is intrinsically subcritical. Indeed, in this case it is possible to define the change of variables $t(r)=\int_{R_1}^r \xi^{-\frac{N-1}{p-1}}d\xi$ that allows to reduce problem \eqref{main} to 
$$
\begin{cases}
-\Delta_p w=a(t) f(w) & \mbox{ in } (0,T), \\
w>0& \mbox{ in } (0,T), \\
w'(0)=w'(T)=0,
\end{cases}\qquad a(t)=r(t)^{\frac{p(N-1)}{p-1}},
$$
where $r(t)$ is the inverse of $t(r)$ (cf. \cite[Remark 2.5]{ABF}). So, the unknown $w(t)=u(r(t))$ must solve a $p$-Laplacian equation which is very similar to the one for $u$ (apart from the weight $a(t)$ which by the way is positive and bounded) and, being 1-dimensional, is always subcritical. 

\medskip

\noindent (iii) The prototype function defined in \eqref{eq:prototype} clearly satisfies $(f_\mathrm{reg})$, $(f_\mathrm{eq})$, and $(f_0)$. Moreover, in the case the domain is a ball, it further satisfies $(f_\mathrm{subc})$: being $r<q<p^*$, there exists $\eta\in (0,1)$ such that $q<\eta^qp^*$, hence
\[
\limsup_{s\to+\infty} \frac{s^q-s^r}{(\eta s)^q/q-(\eta s)^r/r}<p^*.
\]
Similarly, we have that $(f_\mathrm{subc})$ is satisfied whenever $f$ behaves asymptotically (as $s\to+\infty$) as the prototype function \eqref{eq:prototype}, so that assumption $(f_\mathrm{subc})$ allows for a broad class of nonlinearities.

\medskip

\noindent (iv) We note in passing that it is also possible to modify conditions $(f_\mathrm{reg})$ and $(f_\mathrm{eq})$ in such a way to allow the nonlinearity $f$ to have more than one positive zero, we refer to Remark \ref{rem:+zeri} for more details. This corresponds, for problem \eqref{main}, to admit more than one constant solution. Needless to say that the number 1 appearing in condition $(f_\mathrm{eq})$
can be replaced by any $u_0>0$ for which $u\equiv u_0$ is a constant solution of the problem. Its exact value does not play any role. \hfill $\diamond$
\end{remark}

We are now ready to state the main results of the paper. 
We recall that $W^{1,p}$-solutions of \eqref{main} are of class $C^{1,\gamma}(\bar\Omega)$ for some $\gamma>0$, see \cite[Theorem 2]{Lieberman}.
Our first result is an a priori $C^1$-bound for radial solutions of \eqref{main}, either in an annulus under hypotheses $(f_{\mathrm{reg}})$, $(f_{\mathrm{eq}})$, and $(f_0)$, or in a ball under the additional assumption $(f_\mathrm{subl})$ or $(f_\mathrm{subc})$.

\begin{theorem}\label{th:bounds}
Let $\Omega$ be either the annulus $\mathcal A(R_1,R_2)$ or the ball $\mathcal B(R_2)$ and let $f$ satisfy $(f_{\mathrm{reg}})$, $(f_\mathrm{eq})$ and $(f_0)$. 
In the case $\Omega=\mathcal B(R_2)$ assume in addition either $(f_\mathrm{subl})$ or $(f_\mathrm{subc})$.
Then there exists a constant $C=C(f,\Omega,p)>0$ such that every radial solution $u$ of \eqref{main} satisfies
\[
\|u\|_{C^1(\Omega)}\leq C.
\]
\end{theorem}
For the semilinear case $p=2$, when $f$ is the prototype subcritical nonlinearity with $r=2$, some a priori estimates are also proved in \cite[Section 2]{bonheure2016multiple}. The authors find $L^\infty$ and $H^1$-bounds for the solutions of \eqref{main} when $\Omega$ is a general bounded domain. 
As already mentioned, they also obtain $C^1$-estimates for radial solutions $u$ on a ball, for possibly supercritical prototype nonlinearites, under the additional assumption $u(0)<1$.
We remark that our result applies to any radial solution of \eqref{main}, regardless of their value at zero, and includes the case of more general nonlinearities $f$ and the case $p\neq 2$. 
For a priori $L^\infty$-estimates of positive solutions to similar subcritical problems under Dirichlet boundary conditions, we refer for instance to \cite{GidasSpruck,McKennaReichel,CastroPardo} for the semilinear case, to \cite{AzizehClement,Ruiz,Zou_apriori,Dong} and references therein for the quasilinear case. 

In order to state our existence results, we introduce $\lambda_{k}^{\textnormal{rad}}$ as the $k$-th radial eigenvalue of $-\Delta_p u = \lambda|u|^{p-2}u$ in $\Omega$ with Neumann boundary conditions; moreover, we further assume
\begin{itemize}
\item[$(f_1)$] there exists $C_1 := \lim_{s\to1} \frac{f(s)}{\vert s - 1 \vert^{p-2} (s-1)}$.
\end{itemize}
Notice that, by $(f_\textrm{eq})$, it holds $C_1 \in [0,+\infty]$; the differentiability of $f$ at $s =1$ implies that $C_1 = f'(1) \in [0,+\infty)$ if $p = 2$ and $C_1 = 0$ if $1 < p < 2$.

Hereafter, in order to treat simultaneously the cases of the annulus and of the ball, we adopt the convention $R_1=0$ when $\Omega=\mathcal B(R_2)$. Furthermore, since we are interested only in radial solutions, with abuse of notation we write $u(r)=u(x)$ for $|x|=r$.

\begin{theorem}\label{th:C1>0}
Under the same assumptions as in Theorem \ref{th:bounds}, suppose that $(f_1)$ holds with $C_1 > \lambda_{k+1}^{\textnormal{rad}}$ for some integer $k \geq 1$. Then there exist at least $2k$ distinct non-constant radial solutions $u_1,\ldots,u_{2k}$ to \eqref{main}. 
Moreover, we have
\begin{itemize}
\item[(i)] $u_j(R_1)>1$ for every $j=1,\ldots,k$;
\item[(ii)] $u_j(R_1)<1$ for every $j=k+1,\ldots,2k$;
\item[(ii)] $u_j(r)-1$ and $u_{j+k}(r)-1$ have exactly $j$ zeros for $r\in(R_1,R_2)$, for every $j=1,\ldots,k$.
\end{itemize}
In particular, if $C_1 = +\infty$, then \eqref{main} has infinitely many distinct non-constant radial solutions satisfying $u_j(R_1)>1$ and infinitely many satisfying $u_j(R_1)<1$.
\end{theorem}

\begin{theorem}\label{th:C1=0}
Under the same assumptions as in Theorem \ref{th:bounds}, suppose that $(f_1)$ holds with $C_1 = 0$. Then
for any integer $k \geq 1$ and any $\varepsilon > 0$ there exists $R_*(k,\varepsilon) > 0$ such that if 
$$
R_1<\varepsilon R_2\quad\mbox{and}\quad R_2 > R_*(k,\varepsilon),
$$
then problem \eqref{main} has at least $4k$ distinct non-constant radial solutions.

Denoting these solutions by $u_1^+,\ldots,u_{2k}^+$, $u_1^-,\ldots,u_{2k}^-$, we have that
\begin{itemize}
\item[(i)] $u_j^\pm(R_1)>1$ for every $j=1,\ldots,k$;
\item[(ii)] $u_j^\pm(R_1)<1$ for every $j=k+1,\ldots,2k$;
\item[(ii)] $u_j^\pm(r)-1$ and $u_{j+k}^\pm(r)-1$ have exactly $j$ zeros for $r\in(R_1,R_2)$, for every $j=1,\ldots,k$.
\end{itemize}
\end{theorem} 

We observe that the condition $0=R_1 < \varepsilon R_2$ is satisfied for every $\varepsilon>0$ when $\Omega=\mathcal B(R_2)$, hence in this case $R_*=R_*(k)$ can be chosen to be independent of $\varepsilon$. 

Let us now briefly comment on the shape of the solutions found in Theorems~\ref{th:C1>0} and \ref{th:C1=0}. Solutions indexed from 1 to $k$ are above 1 at $r=R_1$ and start decreasing, whereas the other solutions start below the value 1 and in an increasing way. In Theorem~\ref{th:C1=0} there are two solutions having both the same monotonicity at $R_1$ and the same number of oscillations around the constant solution 1. We distinguish them with the symbol $\pm$ which is meant to describe the distance of $u_j^\pm(R_1)$ from 1:  $|u_j^+(R_1)-1|>|u_j^-(R_1)-1|$.

We notice that Theorems~\ref{th:C1>0} and \ref{th:C1=0} are almost complementary, in the following sense (for simplicity, we focus here on the case of the ball $\Omega=\mathcal B(R_2)$). Whenever $C_1 > 0$, recalling that $\lambda_k^{\mathrm{rad}}(\mathcal B(R_2)) \to 0^+$ for $R_2 \to +\infty$, Theorem \ref{th:C1>0} yields the existence of $2k$ radial solutions for $R_2$ large enough. In the same flavor, if $C_1 = 0$ Theorem \ref{th:C1=0} gives the existence of $4k$ radial solutions for $R_2 > R^*(k)$. Actually an intermediate result, giving the existence of $4k - 2(l-1)$ radial solutions for some $l=1,\ldots,k$, is possible depending on the precise value of the constant $C_1$; we refer to Remark \ref{rem:C1=0} for the precise statement.

The existence and the oscillatory behavior of solutions with $u(R_1)<1$, namely solutions indexed from $k+1$ to $2k$ in Theorems \ref{th:C1>0} and \ref{th:C1=0}, have already been proved in \cite[Theorems 1.2 and 1.4]{ABF} respectively, even for possibly supercritical $f$. 
This is the reason why, in Section \ref{Sec:main} below, we focus only on solutions with $u(R_1)>1$.
As already noticed, the existence of such solutions seems to be closely related to the subcriticality assumption required on $f$ in the present paper (see the next section of the Introduction for a technical explanation of this). 

Taking into account that the prototype nonlinearity \eqref{eq:prototype}
satisfies $(f_1)$ with 
$$
C_1=\left\{\begin{array}{ll}
0 & \text{ if } 1<p<2, \\
q-r & \text{ if } p=2, \\
+\infty & \text{ if } p>2, \\
\end{array}\right.
$$
we have the following corollary of Theorems \ref{th:C1>0} and \ref{th:C1=0}. We observe that this result is coherent with the numerical simulations in \cite[Section 3]{ABF}.
 
\begin{corollary}\label{cor:modello}
For $1<p<+\infty$ consider the Neumann problem
\begin{equation}\label{modello}
\left\{
\begin{array}{ll}
\vspace{0.1cm}
-\Delta_p u + u^{r-1} = u^{q-1}  & \mbox{ in } \Omega, \\
\vspace{0.1cm}
u > 0 & \mbox{ in } \Omega, \\
\partial_\nu u = 0 & \mbox{ on } \partial\Omega, \\
\end{array}
\right.
\end{equation}
with $p\le r<q$ in the case $\Omega=\mathcal{A}(R_1,R_2)$, and $p\le r <q<p^*$ in the case $\Omega=\mathcal B(R_2)$. Then:
\begin{itemize}
\item[(i)] for $p > 2$, \eqref{modello} has infinitely many distinct non-constant radial solutions satisfying $u(R_1)>1$ and infinitely many satisfying $u(R_1)<1$;
\item[(ii)] for $p = 2$ and $q-r > \lambda_{k+1}^{\textnormal{rad}}$ for some $k\ge1$, \eqref{modello} has at least $k$ non-constant radial solutions with $u(R_1)>1$ and $k$ non-constant radial solutions with $u(R_1)<1$;
\item[(iii)] for $1 < p < 2$, 
for any integer $k \geq 1$ and any $\varepsilon > 0$ there exists $R_*(k,\varepsilon) > 0$ such that if $R_1 < \varepsilon R_2$ and $R_2 > R_*(k,\varepsilon)$, then problem \eqref{modello} has at least $2k$ non-constant radial solutions with $u(R_1)>1$ and $2k$ non-constant radial solutions with $u(R_1)<1$.
\end{itemize}
\end{corollary}

To conclude this section we would like to mention that a possible future direction of research is the investigation of solutions with $u(R_1)>1$ in the critical case.
Some results in this direction are already contained in the classical papers \cite{AdimurthiYadava,AdimurthiYadava97}, where the shooting method (after having converted the equation in \eqref{main} into an equivalent one via the Emden-Fowler transformation) is indeed used to study the behavior of solutions with $u(R_1)$ large. 
A possible related paper is \cite{CastroKurepa-critical}, where a technique quite similar to the one we use is employed to provide energy estimates for a Dirichlet critical problem. We also mention \cite{delPinoWei}, where the authors consider a semilinear Neumann problem with an exponential nonlinearity in a bounded domain of $\mathbb R^2$ which is possibly non-radial. The techniques therein are quite different, since the authors use the Lyapunov-Schmidt reduction method. As already noticed, the numerical simulations suggest that it is probably a very difficult task to prove existence and properties of solutions with $u(R_1)>1$ to supercritical problems.

\subsection{Ideas of the proofs and organization of the paper}
The proof of Theorem~\ref{th:bounds} essentially relies on an elastic-type property that holds for radial solutions of \eqref{main}. This property is the core of the paper and is fundamental for both proving the a priori bound and the existence of solutions with $u(R_1)>1$; in its proof we strongly use the subcriticality of $f$.
To explain this property, we consider a radial solution $u=u(r)$ of \eqref{main}, with $u(R_1)>0$ and we call $r_1\in (R_1, R_2]$ the biggest radius for which $u$ is positive in $[R_1,r_1)$. 
The elastic-type property says that if $u(R_1)$ is large, then $u(r) + \vert u'(r) \vert$ in $[R_1,r_1)$ is also large. 
When $\Omega$ is an annulus or when $\lim_{s\to+\infty}\int_1^s f(\sigma)d\sigma$ is finite, the proof of this property relies on an energy analysis in the phase plane:  it is a quite simple consequence of the fact that the energy of radial solutions is non-increasing (cf. Proposition \ref{prop:explosion1}). In the remaining case, i.e., when $\Omega$ is a ball and $\lim_{s\to+\infty}\int_1^s f(\sigma)d\sigma=+\infty$, the proof is more involved and, in order to perform the energy analysis, we need to derive some identities \`a la Pohozaev and Pucci-Serrin. The technique used was first introduced by Castro and Kurepa in \cite{CastroKurepa} for the Laplacian, then generalized to the $p$-Laplacian in \cite{ElHachimiThelin}, and finally refined and further generalized to non-homogeneous $p$-Laplacian-like operators in \cite{GHMZ}. 
We take inspiration mainly from the latter by Garc{\'\i}a-Huidobro, Man{\'a}sevich, and Zanolin, even though this article, as all the previous ones, deals with the Dirichlet problem and the equation is of the type $Lu = g(u) + V(\vert x \vert)$ (here, $L$ is the differential operator) with $g: \mathbb{R} \to \mathbb{R}$ satisfying subcriticality assumptions at $\pm \infty$; therefore some delicate modifications of the arguments therein are needed to adapt this technique to our situation.

For our goals, the most important consequence of this elastic property is that radial solutions of \eqref{main} cannot have $u(R_1)$ too large, see Proposition~\ref{Pr:d*}. This fact, together with the monotonicity of the energy, are the main ingredients of the proof of Theorem \ref{th:bounds}. 

The technique used to prove Theorems \ref{th:C1>0} and \ref{th:C1=0} is, as in \cite{ABF}, the shooting method. This approach is classical in the qualitative theory of ODEs; we mention here the papers \cite{MontefuscoPucci,BarutelloSecchiSerra} where the shooting method is used for proving existence of solutions to radial semilinear and quasilinear problems in some related situations. The idea of the proof is essentially the same as in \cite{ABF}, see also \cite{FBproceedings}. Since we are interested only in radial solutions, we rewrite the problem in the one-dimensional radial form. In turn, the one-dimensional second-order equation can be seen as the following first-order planar system 
\begin{equation}\label{planar-system}
u'=r^{-(N-1)(p'-1)}|v|^{p'-2}v,\quad v'=-r^{N-1}f(u),\qquad r\in(R_1,R_2),
\end{equation}
where $1/p'=1-1/p$.
Here the shooting method begins: instead of looking for solutions of the system that satisfy Neumann boundary conditions, we impose the initial conditions 
$$u(R_1)=d,\quad v(R_1)=0\quad\mbox{ with }d\in(0,+\infty),$$ 
and look for initial data $(d,0)$ such that the corresponding solution $(u_d,v_d)$ to the Cauchy problem satisfies $v(R_2)=0$. 
In Section \ref{Sec:shooting}, we recall global existence, uniqueness and continuous dependence on the initial data for the Cauchy problem. In particular, uniqueness implies that if $d\neq 1$, the solution $(u_d,v_d)$ is such that $(u_d(r),v_d(r)) \neq (1,0)$ for every $r\in[R_1,R_2]$. Therefore, it is possible to pass to polar-like coordinates about the point $(1,0)$ in the phase plane $(u,v)$. Furthermore, the assumption $(f_\mathrm{eq})$ and the equation guarantee that the solutions of the Cauchy problem wind clockwise around $(1,0)$. We look for solutions that make an integer number of half-turns around $(1,0)$ in the phase plane $(u,v)$. Hence, with this scheme in mind, all that we have to do is to count the number of turns performed by the solutions around $(1,0)$. 
To this aim, when $C_1>\lambda_{k+1}^\mathrm{rad}$, we estimate the number of turns of the solutions shot from $d$ close to 1, in terms of the number of turns of the $(k+1)$-th radial eigenfunction of the Neumann $p$-Laplacian eigenvalue problem. In this way, we show that the solutions corresponding to $d$ close to 1 perform more than $k$ half-turns.  
On the other hand, as a consequence of the elastic-type property, we know that for $d$ above a certain threshold $d^*$, the solution $(u_d,v_d)$ of the Cauchy problem performs less than one half-turn. By the continuous dependence on $d$, there must exist $k$ values of $d$, $1<d_k<d_{k-1}<\dots< d_1<d^*$, to which correspond the solutions of Theorem~\ref{th:C1>0}, cf. Figure \ref{fig:C1>0}. We stress that in this argument it is essential to have the threshold $d^*$.
\begin{figure}[h!t]
\begin{center}
\includegraphics[scale=.8]{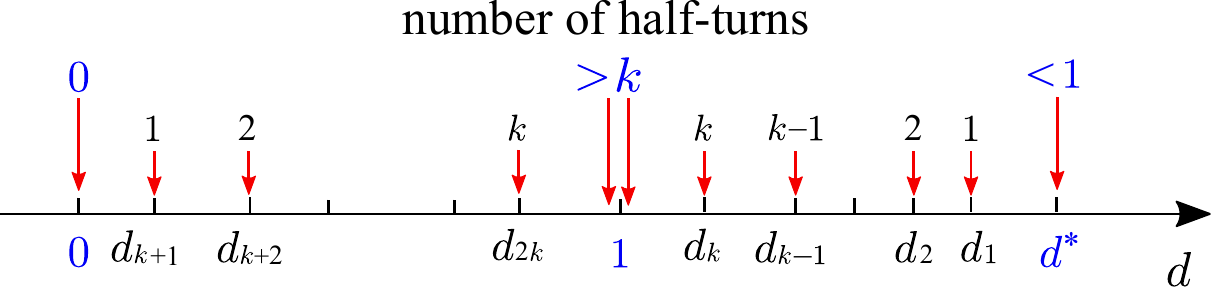}
\caption{Qualitative representation, in the case $C_1>\lambda_{k+1}^{\mathrm{rad}}$, of the initial data  $d_1,\,d_2,\dots,d_k \in(1,d^*)$ to which correspond the solutions $(u_{d_{j}},v_{d_{j}})$, with $u_{d_{j}}(R_1)=d_j>1$, performing exactly $j$ clockwise half-turns around $(1,0)$ in the phase plane for every $j=1,\dots,k$.   Analogously, to the initial data $d_{k+1},\,d_{k+2},\dots,d_{2k} \in (0,1)$ represented in the picture, correspond the solutions $(u_{d_{j+k}},v_{d_{j+k}})$, with $u_{d_{j+k}}(R_1)=d_{j+k}<1$, performing exactly $j$ clockwise half-turns around $(1,0)$ in the phase plane for every $j=1,\dots,k$.
The solutions whose existence is stated in Theorem \ref{th:C1>0} are $u_j=u_{d_j}$ and $u_{j+k}=u_{d_{j+k}}$ for $j=1,\dots,k$.}\label{fig:C1>0}
\end{center}
\end{figure}

When $C_1=0$ the proof is complicated by the fact that both in a neighborhood of 1 and in a neighborhood of $d^*$ the solutions perform less than one half-turn. Nevertheless, by means of an argument introduced in \cite{BosZan13} (see Proposition~\ref{prop:d_k} below for a more detailed description), we are able to find a $\hat d_k\in(1,d^*)$  such that $(u_{\hat d_k},v_{\hat d_k})$ performs more than $k$ half-turns. Hence, the continuous dependence argument can be used both in $(1,\hat d_k)$ and in $(\hat d_k,d^*)$ to get in total $2k$ solutions with $u(R_1)>1$, cf. Figure \ref{fig:C1=0}.
\begin{figure}[h!t]
\begin{center}
\includegraphics[scale=.8]{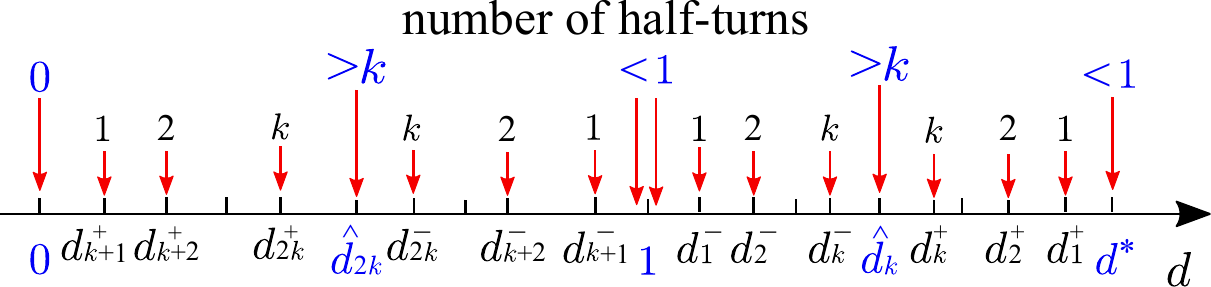}
\caption{Qualitative representation, in the case $C_1=0$, of the initial data $d^\pm_1,\,d^\pm_2,\dots,d^\pm_k \in(1,d^*)$ to which correspond the solutions $(u_{d^\pm_{j}},v_{d^\pm_{j}})$, with $u_{d^\pm_{j}}(R_1)=d^\pm_{j}>1$, performing exactly $j$ clockwise half-turns around $(1,0)$ in the phase plane for every $j=1,\dots,k$. The value $\hat d_k\in(1,d^*)$ is the initial datum for which the number of half-turns of the corresponding solution is greater than $k$. All data $d^+_{j}$ are on the right of $\hat d_k$, while the data $d^-_j$ are on the left of the same value.
Analogously, to the initial data $d^\pm_{k+1},\,d^\pm_{k+2},\dots,d^\pm_{2k} \in (0,1)$ represented in the picture, correspond the solutions $(u_{d^\pm_{j+k}},v^\pm_{d_{j+k}})$, with $u^\pm_{d_{j+k}}(R_1)=d^\pm_{j+k}<1$, performing exactly $j$ clockwise half-turns around $(1,0)$ in the phase plane for every $j=1,\dots,k$. Here we have denoted with $\hat d_{2k}\in(0,1)$ the initial datum for which the number of half-turns of the corresponding solution is greater than $k$. The situation is specular with respect to the case $d>1$. 
In both cases, the $-$ data are closer to 1 than the + ones.
The solutions whose existence is stated in Theorem~\ref{th:C1=0} are $u^\pm_j=u_{d^\pm_j}$ and $u^\pm_{j+k}=u_{d^\pm_{j+k}}$ for $j=1,\dots,k$.}\label{fig:C1=0}
\end{center}
\end{figure}

To conclude the Introduction, we would like to mention that other techniques have already been used to attack similar problems set in a ball of $\mathbb R^N$. In \cite{bonheure2016multiple, ma2016bonheure} the semilinear case is studied by means of the bifurcation theory of Crandall and Rabinowitz. Moreover, variational methods are used for proving the existence of an increasing solution in the semilinear case (cf. \cite{BNW,SerraTilli2011}) and in the quasilinear case (cf. \cite{secchi2012increasing,ColasuonnoNoris}, see also \cite{Fproceedings}) for general possibily supercritical nonlinearities. In particular, the techinque used in \cite[Section 4]{BNW} for $p=2$ and in \cite{ColasuonnoNoris} for $p>2$ can be applied also to annular domains, in this case it provides the existence of at least two monotone solutions, one increasing and one decreasing, cf. also \cite[Section 3]{BonheureGrossiNorisTerracini2015}.   

The paper is organized as follows. In Section \ref{Sec:shooting} we describe the shooting approach and recall some useful properties of the Cauchy problem associated to \eqref{planar-system}. Section~\ref{Sec:elastic} is entirely devoted to the proof of the elastic-type property, while Section~\ref{Sec:main} contains the proofs of Theorems~\ref{th:bounds}, \ref{th:C1>0} and \ref{th:C1=0}, as well as some hints for possible variants of our main results. In a final Appendix, we give for the reader's convenience the proof of a technical result (Proposition \ref{prop:d_k}) used along the proof of Theorem \ref{th:C1=0}.

\section{The shooting approach}\label{Sec:shooting}
In the rest of the paper we assume that $\Omega$ is either the annulus $\mathcal A(R_1,R_2)$ or the ball $\mathcal B(R_2)$, with the convention that $R_1=0$ in the case $\Omega=\mathcal B(R_2)$, and we only consider radial solutions of \eqref{main}.
We also suppose, without mentioning it explicitly, that $f$ satisfies $(f_{\mathrm{reg}})$, $(f_\mathrm{eq})$, and $(f_0)$.

Let us introduce a continuous extension $\hat f: \mathbb{R} \to \mathbb{R}$ of $f$ defined as follows
$$
\hat f(s)=\begin{cases}
	f(s) \quad&\text{if } s \geq 0,\\
	0 &\text{if } s < 0.
\end{cases}
$$
Writing the p-Laplacian operator in radial form, consider the following problem involving $\hat f$
\begin{equation}\label{ode2}
\begin{cases}
-\left(r^{N-1}\varphi_p(u')\right)'=r^{N-1}\hat f(u)\quad\mbox{in }(R_1,R_2),\\
u'(R_1)=u'(R_2)=0,
\end{cases}
\end{equation}
where $\varphi_p(s) = \vert s \vert^{p-2} s$ and the prime symbol $'$ denotes the derivative with respect to $r$. A maximum principle-type result proved in \cite{ABF} ensures that we can study problem \eqref{main} by looking for non-constant solutions of \eqref{ode2}.

\begin{lemma}{\cite[Lemma 2.1]{ABF}}\label{lem:rad}
$u$ is a radial solution of \eqref{main} if and only if $u$ solves \eqref{ode2} and $u\not\equiv -C$, with $C\geq0$.
\end{lemma}

Proceeding as in \cite{ABF}, we adopt the shooting technique: for any $d\geq 0$ we consider the couple $(u_d,v_d)$ that is the unique solution of
\begin{equation}\label{eq:shooting}
\begin{cases}
u_d' = \varphi_p^{-1}\left(\frac{v_d}{r^{N-1}}\right) \quad & r\in (R_1,R_2)\\
v'_d = - r^{N-1}\hat f(u_d) \quad & r\in (R_1,R_2) \\
u_d(R_1) = d \\
v_d(R_1) = 0.
\end{cases}
\end{equation}
The uniqueness, global continuability and continuous dependence for \eqref{eq:shooting} have been proved in \cite{ReiWal97}, we refer to \cite[Lemma 2.2]{ABF} for further details. The last mentioned lemma is stated for $d\in [0,1]$, but the proof holds the same for any $d\geq0$; we warn the reader that the notation therein is different.

\begin{lemma}{\cite[Lemma 2.2]{ABF}}\label{lem:cont_dep}
For any $d\geq 0$, the solution $(u_d,v_d)$ of \eqref{eq:shooting} is unique
and can be defined on the whole $[R_1,R_2]$; moreover, 
if $(d_n)_{n}\subset[0,+\infty)$ is such that $d_n \to d \in [0,+\infty)$ as $n\to+\infty$, then $(u_{d_n}(r),v_{d_n}(r)) \to (u_d(r),v_d(r))$ uniformly for $r \in [R_1,R_2]$.
\end{lemma}

The function $u_d$ solves \eqref{ode2} if and only if $(u_d,v_d)$ is such that $v_d(R_2) = 0$. Incidentally, notice that $(u_1(r),v_1(r)) \equiv (1,0)$, corresponding to the constant solution $1$ of \eqref{ode2}. To find solutions with $d \neq 1$, we rewrite, 
as a consequence of the uniqueness in Lemma \ref{lem:cont_dep}, system \eqref{eq:shooting} using the following polar-like coordinates around the point $(1,0)$ 
\begin{equation}\label{polari}
\begin{cases}
u(r)-1=\rho(r)^\frac{2}{p} \cos_p(\theta(r))\\
v(r)=-\rho(r)^\frac{2}{p'} \sin_p(\theta(r)),
\end{cases}
\end{equation}
where $(\cos_p,\sin_p)$ is the unique solution of the system $x'=-\varphi_{p'}(y)$, $y'=\varphi_p(x)$ with initial conditions $x(0)=1, \ y(0)=0$.
We refer to \cite{DelPinoElguetaManasevich89,DelPinoManasevichMurua92,FabryFayyad92} and to \cite[Lemma 2.3]{ABF} for some useful properties of the $p$-cosine and $p$-sine functions.
Via the change of coordinates \eqref{polari}, system \eqref{eq:shooting} is transformed into
\begin{equation}\label{sys2}
\left\{
\begin{array}{l}
\displaystyle \rho'=\frac{p}{2\rho}\, u' \, \left[\varphi_p(u-1)-r^{(N-1)p'} \hat f(u) \right] \vspace{0.2cm}  \\
\displaystyle \theta'=\frac{r^{N-1}}{\rho^2}\left[ (p-1)|u'|^p+(u-1) \hat f(u) \right],
\end{array}
\right.
\end{equation}
with the initial conditions 
\begin{equation}\label{ci2}
\begin{cases}
\rho(R_1) =(1-d)^{\frac{p}2}, \quad \theta(R_1) = \pi_p&\quad\mbox{if }d<1,\\
\rho(R_1) =(d-1)^{\frac{p}2}, \quad \theta(R_1) = 0&\quad\mbox{if }d>1.
\end{cases}
\end{equation}
We denote the corresponding solution by $(\rho_d,\theta_d)$. Clearly, the couple $(\rho_d,\theta_d)$ gives rise to a solution 
$u_d(r) = 1 + \rho_d(r)^{\frac{p}{2}} \cos_p(\theta_d(r))$ of \eqref{ode2} (and in turn of \eqref{main}, since the case of a negative constant is ruled out by $u_d(0) = d \geq 0$) if and only if $\theta_d(R_2)=k\pi_p$ for some $k\in\Z$, that is, if and only if the solution $(u_d,v_d)$ performs an integer number of half-turns around the point $(1,0)$; incidentally, note that such rotations always take place in the clockwise sense, since by \eqref{sys2} and $(f_\textrm{eq})$, the function $\theta_d(r)$ is monotone increasing.
For further convenience, we also observe that, by \eqref{polari},
\begin{equation}\label{eq:polar_relation}
r^{(N-1)p'}|u'|^p=\rho^2|\sin_p(\theta)|^{p'} \quad\text{and}\quad |u-1|^p=\rho^2|\cos_p(\theta)|^p.
\end{equation}
As an immediate consequence of Lemma \ref{lem:cont_dep}, \eqref{polari} and \cite[Lemma 2.3]{ABF}, we have the following.

\begin{corollary}\label{cor:limit_rho_as_d_to0}
If $(d_n)\subset[0,1)\cup(1,+\infty)$ is such that $d_n \to d \in [0,1)\cup(1,+\infty)$, then 
$$(\rho_{d_n}(r),\theta_{d_n}(r)) \to (\rho_d(r),\theta_d(r))\quad \mbox{uniformly in }r \in [R_1,R_2].$$ Furthermore, 
\begin{equation}\label{2.14}
\lim_{d\to1} \sup_{r\in [R_1,R_2]} \rho_d(r)=0.
\end{equation}
\end{corollary}

In the rest of the section we recall some known results concerning the related radial eigenvalue problem 
\begin{equation}\label{eq:eigenvalue}
\begin{cases}
-(r^{N-1}\varphi_p(u'))'=\lambda r^{N-1} \varphi_p(u) \quad & \mbox{in } (R_1,R_2) \\
u'(R_1)=u'(R_2)=0,
\end{cases}
\end{equation}
and infer some information concerning $\theta_d(R_2)$ when $d \to 1$ by applying the Comparison Theorem to systems \eqref{ode2} and \eqref{eq:eigenvalue}.
\begin{theorem}{\cite[Theorem 1]{RW99}}\label{th:eigen}
The eigenvalue problem \eqref{eq:eigenvalue} has a countable number of simple eigenvalues $0=\lambda_1^{\mathrm{rad}}<\lambda_2^{\mathrm{rad}}<\lambda_3^{\mathrm{rad}}<\dots$, $\lim_{k\to+\infty}\lambda_k^{\mathrm{rad}}=+\infty$, and no other eigenvalues. The eigenfunction that corresponds to the $k$-th eigenvalue $\lambda_k^{\mathrm{rad}}$ has $k-1$ simple zeros in $(R_1,R_2)$.  
\end{theorem}
Via the change of variables 
\begin{equation}\label{polar-eigenprobl}
\begin{cases}
u(r)=\varrho_\lambda(r)^{\frac{2}{p}}\cos_p(\vartheta_\lambda(r))\\
r^{N-1}\varphi_p(u'(r))=-\varrho_\lambda(r)^{\frac{2}{p'}}\sin_p(\vartheta_\lambda(r)),
\end{cases}
\end{equation}
we can rewrite the equation in \eqref{eq:eigenvalue} as 
$$
\begin{cases}
\displaystyle \varrho_\lambda'=\frac{p}{2\varrho}\left(1-\lambda r^{(N-1)p'}\right)\varphi_p(u)u',
\vspace{0.2cm}\\
\displaystyle \vartheta_\lambda'=\frac{r^{N-1}}{\varrho^2}\left[(p-1)|u'|^p+\lambda |u|^p\right].
\end{cases}
$$
Notice that the function $r \mapsto \vartheta_\lambda(r)$ is strictly increasing. As a consequence, if $\lambda=\lambda^{\mathrm{rad}}_{k+1}$ for $k\ge0$, the fact that the eigenfunction which corresponds to the $(k+1)$-th eigenvalue has $k$ simple zeros in $(R_1,R_2)$ reads as 
\begin{equation}\label{initialcondition-ep}
\vartheta_{\lambda^{\mathrm{rad}}_{k+1}}(R_2) - \vartheta_{\lambda^{\mathrm{rad}}_{k+1}}(R_1) =k\pi_p.
\end{equation}
Proceeding as in \eqref{eq:polar_relation}, we have that \eqref{polar-eigenprobl} implies
$$
r^{(N-1)p'}|u'|^p=\varrho_\lambda^2|\sin_p(\vartheta_\lambda)|^{p'}\quad\mbox{ and }\quad|u|^p=\varrho_\lambda^2|\cos_p(\vartheta_\lambda)|^p,
$$
so that
\begin{equation}\label{eq:theta'-omogenea-associata}
\vartheta_\lambda' =r^{N-1}\left[\frac{p-1}{r^{(N-1)p'}}|\sin_p(\vartheta_\lambda)|^{p'}+\lambda|\cos_p(\vartheta_\lambda)|^p\right].
\end{equation}

\begin{lemma}\label{lem:d_close1}
If, for some integer $k\geq1$,
\[
\liminf_{s\to1} \frac{f(s)}{\varphi_p(s-1)}>\lambda_{k+1}^{\textnormal{rad}}
\qquad \left(\text{respectively, } \; \limsup_{s\to1} \frac{f(s)}{\varphi_p(s-1)}< \lambda_{k+1}^{\textnormal{rad}} \right), 
\]
then there exists $\bar\delta>0$ such that $\theta_d(R_2) - \theta_d(R_1)>k\pi_p$ (respectively, $<$) for $d\in [1-\bar\delta,1+\bar\delta]$ and $d \neq 1$.
\end{lemma}
\begin{proof}
Suppose that $\liminf_{s\to1} \frac{f(s)}{\varphi_p(s-1)}>\lambda_{k+1}^{\mathrm{rad}}$.
There exists $\delta>0$ such that for every $s$ satisfying $|s-1|<\delta$ it holds
$$
\hat f(s)(s-1)=f(s)(s-1)>
\lambda_{k+1}^{\mathrm{rad}} |s-1|^p.
$$
Then, by \eqref{sys2}, we get that if $|u(r)-1|<\delta$,
$$
\theta'_d(r)>
\frac{r^{N-1}}{\rho(r)^2}\left[(p-1)|u'(r)|^p+\lambda_{k+1}^{\mathrm{rad}}|u(r)-1|^p\right].
$$
Combining the latter inequality with \eqref{eq:polar_relation} and \eqref{2.14}, we obtain  that there exists $\bar\delta>0$ such that for all $d\in (1-\bar\delta,1+\bar\delta)$ with $d \neq 1$
$$
\theta_d'(r)> r^{N-1}\left[\frac{p-1}{r^{(N-1)p'}}|\sin_p(\theta_d(r))|^{p'}+\lambda_{k+1}^{\mathrm{rad}} |\cos_p(\theta_d(r))|^p\right]
$$
for all $r\in (R_1,R_2]$.
Recalling \eqref{eq:theta'-omogenea-associata} with $\lambda=\lambda^{\mathrm{rad}}_{k+1}$ and using the Comparison Theorem for ODEs we obtain, for all $d\in (1-\bar\delta,1+\bar\delta)$ with $d \neq 1$,

$$
\theta_d(r) - \theta_d(R_1)> \vartheta_{\lambda^{\mathrm{rad}}_{k+1}}(r)- \vartheta_{\lambda^{\mathrm{rad}}_{k+1}}(R_1)\quad\mbox{for all }r\in(R_1,R_2].
$$
In particular, by \eqref{initialcondition-ep}, $\theta_d(R_2) - \theta_d(R_1)> k\pi_p$. The remaining case can be treated in the same way.
\end{proof}

\section{An elastic-type property}\label{Sec:elastic}
In what follows we suppose that $f$ satisfies $(f_{\mathrm{reg}})$, $(f_{\mathrm{eq}})$, $(f_0)$, and that, in the case $\Omega=\mathcal B(R_2)$, it fulfills in addition either $(f_\mathrm{subl})$ or $(f_\mathrm{subc})$.

The main aim of this section is to prove that, under these assumptions, the solution $(u_d,v_d)$ of \eqref{eq:shooting} enjoys the following property
\begin{equation}\label{eq:explosion_d}
\lim_{d\to+\infty} \left( u_d(r)+|u_d'(r)| \right)=+\infty \quad \text{uniformly for } r\in [R_1,r_1(d)],
\end{equation}
where, for $d > 0$, 
\begin{equation}\label{eq:r1def}
r_1(d):=\max\left\{r\in [R_1,R_2]:\ u_d(s)> 0\ \text{ for all } s\in [R_1,r) \right\}.
\end{equation}
Notice that $r_1(d)$ is the first zero of $u_d(r)$ (and, actually, the unique one, since $\hat{f}(s) = 0$ for $s \leq 0$) if any; otherwise, $r_1(d) = R_2$. Following the literature, we call \eqref{eq:explosion_d} an elastic-type property because it says that, whenever $u_d(R_1)$ is large, it follows that the norm of $(u_d(r),u'_d(r))$ is also large, uniformly in $r$, at least as long as $u_d(r) \geq 0$. For the sake of clarity, we also remark that \eqref{eq:explosion_d} explicitly means that
$$
\lim_{d \to +\infty}\min_{r \in [R_1,r_1(d)]}\left( u_d(r)+|u_d'(r)| \right)=+\infty.
$$ 
We will prove this separately in Propositions \ref{prop:explosion1}, \ref{prop:explosion2} and \ref{prop:explosion3}, depending on the hypotheses on $f$ and $\Omega$.

As a crucial tool for most of our next arguments, for any $d>0$ we introduce the energy 
$$
H_d(r):=\frac{|u_d'(r)|^p}{p'}+\hat F(u_d(r)), \qquad r\in [R_1,R_2],
$$
where for every $s\in\R$
\[
\hat F(s):=\int_1^s \hat f(\sigma)\,d\sigma.
\]
In view of $(f_{\mathrm{eq}})$ and of the definition of $\hat f$, it holds that $\hat F(s)\geq0$ for every $s\in\R$ and $\hat F(s)=0$ if and only if $s=1$. Moreover, $\hat F(s)$ is monotone increasing for $s\geq1$, so that
\begin{equation}\label{eq:Finfty_def}
F_\infty:=\lim_{s\to+\infty} F(s)=\lim_{s\to+\infty}\hat F(s)
\end{equation}
is well defined.

We deduce that $H_d(r)\geq0$ for every $d>0$ and $r\in [R_1,R_2]$, and that $H_d(r)=0$ if and only if $u_d(r)=1$ and $u_d'(r)=0$.
Observing that, for $r \neq 0$,
\begin{equation}\label{eq:useful_rel}
\vert u_d'(r) \vert^p = \vert \varphi_p(u_d'(r)) \vert^{p'}, \qquad 
\left(\varphi_p(u_d'(r))\right)' = - \frac{N-1}{r} \varphi_p(u_d'(r)) - \hat f(u_d(r)),
\end{equation}
a straightforward computation yields
\begin{equation}\label{eq:H'}
H'_d(r) = - \frac{N-1}{r}\vert u_d'(r) \vert^p\le0 \qquad \mbox{ for } r \neq 0.
\end{equation}
As a consequence, 
\begin{equation}\label{eq:H}
H_d(r_1(d)) \leq H_d(r) \leq H_d(R_1) \qquad \mbox{ for } r \in [R_1,r_1(d)].
\end{equation}

\begin{remark}\label{equivalence}
It is easy to see that the elastic property \eqref{eq:explosion_d} holds true if   
\begin{equation}\label{eq:explosion_H}
\lim_{d\to+\infty} H_d(r_1(d))=+\infty.
\end{equation}
Indeed, suppose by contradiction that \eqref{eq:explosion_H} holds and that there exist a constant $M>0$ and sequences $d_n\to+\infty$ and $r_n\in [R_1,r_1(d_n)]$ such that 
$$
u_{d_n}(r_n)+|u'_{d_n}(r_n)|\le M\quad\mbox{for all }n.
$$  
Then, recalling \eqref{eq:H} and the fact that $\hat F$ is decreasing in $[0,1)$ and increasing in $[1,+\infty)$, we find
$$
H_{d_n}(r_1(d_n)) \leq H_{d_n}(r_n)\leq M^p/p'+\max\{\hat F(0),\hat F(M)\} \quad\mbox{for all } n,
$$
which contradicts \eqref{eq:explosion_H}. \hfill $\diamond$
\end{remark}

If either $\Omega$ is an annulus or $f$ is integrable at $+\infty$, we can easily prove the elastic property \eqref{eq:explosion_d} as a consequence of relation \eqref{eq:H'}.
 
\begin{proposition}\label{prop:explosion1}
Suppose that one of the following two assumptions holds
\begin{itemize}
\item[(i)] the quantity $F_\infty$ defined in \eqref{eq:Finfty_def} is finite;
\item[(ii)] $\Omega=\mathcal A(R_1,R_2)$, with $R_1>0$.
\end{itemize}
Then \eqref{eq:explosion_d} holds.
\end{proposition}
\begin{proof}
(i) Relation \eqref{eq:H'}, together with the fact that $\hat F$ is decreasing in $[0,1)$ and increasing in $[1,+\infty)$ provides
\[
\frac{|u_d'(r)|^p}{p'} \leq H_d(r)\leq H_d(R_1)=\hat F(d) \leq \max\{\hat F(0),F_\infty\},
\]
for every $d>0$ and $r\in [R_1,R_2]$ (with $R_1=0$ in the case $\Omega=\mathcal{B}({R_2})$). By integrating the previous inequality, we obtain
\[
|u_d(r)-d|\leq (p'\max\{\hat F(0),F_\infty\})^{1/p}(R_2-R_1), 
\]
which, being $F_\infty$ finite, immediately provides that $\lim_{d\to+\infty}u_d(r)=+\infty$ uniformly for $r\in [R_1,R_2]$ and hence \eqref{eq:explosion_d}.

(ii) Suppose that $F_\infty=+\infty$ and that $R_1>0$. Relation \eqref{eq:H'} provides
\[
|H_d'(r)|\leq \frac{N-1}{R_1}|u'_d(r)|^p \leq \frac{(N-1)p'}{R_1}H_d(r),
\]
for every $d>0$ and $r\in [R_1,R_2]$, hence, by Gronwall's lemma,
\[
H_d(r)\geq H_d(R_1)e^{\frac{(N-1)p'}{R_1}(R_1-r)} \geq \hat F(d) e^{\frac{(N-1)p'}{R_1}(R_1-R_2)},
\]
for every $d>0$ and $r\in [R_1,R_2]$.
The assumption $F_\infty=\lim_{d\to+\infty}\hat F(d)=+\infty$ allows to conclude that \eqref{eq:explosion_H} holds and, by Remark \ref{equivalence}, \eqref{eq:explosion_d} holds as well.
\end{proof}

As already mentioned in the Introduction, in order to treat the remaining case $\Omega=\mathcal{B}(R_2)$ and $F_\infty=+\infty$, we take inspiration essentially from \cite{GHMZ}.
From now on in this section, $R_1=0$ and we assume either $(f_{\mathrm{subl}})$ or $(f_{\mathrm{subc}})$. We let 
\begin{equation*}
\eta:=\begin{cases}
\mbox{any number in }(0,1)\mbox{ if }(f_{\mathrm{subl}})\mbox{ holds,} \\
\mbox{the constant introduced in }(f_{\mathrm{subc}})\mbox{ if the latter holds.}
\end{cases}
\end{equation*}
For every $d>0$, we introduce the quantity,
\begin{equation}\label{eq:r0def}
r_0(d):=\max\left\{r\in [0,R_2]:\ u_d(s)>\eta d \ \text{ for all } s\in [0,r) \right\}.
\end{equation}
We notice that 
\begin{equation}\label{eq:r0property}
0<r_0(d)<r_1(d),
\end{equation}
the former inequality descending from $u_d(0)=d>\eta d$ for every $d>0$, and the latter simply because $\eta d>0$.
The following estimate from below of $r_0(d)$ will be crucial in the sequel.

\begin{lemma}\label{lem:r0estimate}
If $d>1/\eta$ then, for every $r\in (0,r_0(d)]$, it holds
\begin{itemize}
\item[(i)] $u_d'(r)<0$;
\item[(ii)] $1<\eta d\leq u_d(r)\leq d $;
\item[(iii)] $0<\hat f(u_d(r))=f(u_d(r)) \leq f^*(d)$.
\end{itemize}
In addition, we have the following estimate from below of $r_0(d)$
\begin{equation}\label{eq:r0estimate}
r_0(d)\geq \min\left\{ R_2 \,; \,
[(1-\eta)dp']^{\frac{1}{p'}}\left(\frac{N}{f^*(d)}\right)^{\frac{1}{p}} \right\}.
\end{equation}
\end{lemma}
\begin{proof}
In order to prove (i), notice that the assumption $d>1/\eta$ implies $u_d(r)\ge \eta d>1$ for every $r\in (0,r_0(d)]$ and hence, by $(f_\textrm{eq})$, $\hat f(u_d(r))>0$ in the same interval. The second equation in \eqref{eq:shooting} then implies $v'_d(r)<0$ in this interval and, by integration, $v_d(r)<0$. Then the first equation in \eqref{eq:shooting} provides (i).

Properties (ii) and (iii) follow immediately from (i).

Now we prove \eqref{eq:r0estimate}. If $u_d(r)>\eta d$ for every $r\in [0,R_2)$, then $r_0(d)= R_2$ and we are done. Otherwise, 
\begin{equation}\label{eq:r0}
r_0(d) \in (0,R_2) \qquad\text{and}\qquad u_d(r_0(d))=\eta d.
\end{equation}
By integrating \eqref{ode2} and inserting (iii) we obtain, for $r\in [0,r_0(d))$,
\[
\varphi_p(u_d'(r))=-\frac{1}{r^{N-1}} \int_0^r t^{N-1} \hat f(u_d(t)) \,dt
\geq -\frac{1}{r^{N-1}} \int_0^r t^{N-1} f^*(d) \,dt =-\frac{r}{N} f^*(d).
\]
Being $\varphi_p$ invertible and $\varphi_p^{-1}$ monotone increasing, this provides
\[
u_d'(r)\geq \varphi_p^{-1}\left(-\frac{r}{N} f^*(d) \right),
\quad \text{for every }r\in [0,r_0(d)).
\]
We integrate again the previous inequality in $(0,r_0(d))$ and use \eqref{eq:r0} to get
\[
(\eta-1)d \geq \int_0^{r_0(d)} \varphi_p^{-1} \left(-\frac{r}{N} f^*(d)\right) \,dr.
\]
Noticing that $\varphi_p^{-1}=\varphi_{p'}$, we deduce
\[
(1-\eta)d\leq \frac{r_0(d)^{p'}}{p'} \left(\frac{f^*(d)}{N}\right)^{p'-1},
\]
which provides \eqref{eq:r0estimate}.
\end{proof}

Using Lemma \ref{lem:r0estimate}, the elastic-type property \eqref{eq:explosion_d} can be quite easily established when 
$f$ satisfies $(f_\mathrm{subl})$. Precisely, we have the following proposition.

\begin{proposition}\label{prop:explosion2}
Suppose that $\Omega=\mathcal B(R_2)$ and $f$ satisfies $(f_{\mathrm{subl}})$. 
Then \eqref{eq:explosion_d} holds.
\end{proposition}
\begin{proof}
We can suppose that $F_\infty=+\infty$, since the complementary case was treated in Proposition \ref{prop:explosion1}-(i).
Using \eqref{eq:H'} we easily obtain
$$
H_d'(r) + \frac{(N-1)p'}{r}H_d(r) \geq 0, \quad \mbox{ for every } r \in (0,R_2].
$$
Multiplying the above inequality by $r^{p'(N-1)}$ we infer that 
$$
\frac{d}{dr} \left( r^{p'(N-1)} H_d(r) \right) \geq 0, \quad \mbox{ for every } r \in (0,R_2],
$$
so that integrating from $r_0(d)$ to $r_1(d)$ (recall \eqref{eq:r0property}) and using that $r_1(d) \leq R_2$ yields
$$
H_d(r_1(d)) \geq R_2^{-p'(N-1)} r_0(d)^{p'(N-1)}H_d(r_0(d)).
$$ 
Using \eqref{eq:r0estimate} and the fact that $H_d(r_0(d)) \geq F(\eta d)$ (which follows from Lemma \ref{lem:r0estimate} (ii) and from the fact that $\hat F(s)$ is increasing for $s\geq1$), we obtain 
$$
H_d(r_1(d)) \geq R_2^{-p'(N-1)} 
\min\left\{ R_2^{p'(N-1)} \,; \,[(1-\eta)dp']^{N-1}\left(\frac{N}{f^*(d)}\right)^{\frac{N-1}{p-1}} \right\} F(\eta d).
$$
From assumption $(f_\mathrm{subl})$, we get $f(d) \leq (M+1) d^{p-1}$ for $d$ large enough; therefore, 
$f^*(d) \leq (M+1) d^{p-1}$ as well. Hence, for $d$ large,
$$
H_d(r_1(d)) \geq R_2^{-p'(N-1)} 
\min\left\{ R_2^{p'(N-1)} \,; \,[(1-\eta)p']^{N-1}\left(\frac{N}{M+1}\right)^{\frac{N-1}{p-1}} \right\} F(\eta d) .
$$
Since $F_\infty = +\infty$, we obtain $H_d(r_1(d)) \to +\infty$ for $d \to +\infty$. By Remark \ref{equivalence} this provides 
\eqref{eq:explosion_d}.
\end{proof}

The case when $f$ satisfies $(f_\mathrm{subc})$ is more delicate and some further work is needed. Below, we state and prove two useful lemmas.

\begin{lemma}\label{lem:Pohozaev}
Every solution of the equation $-\left(r^{N-1}\varphi_p(u')\right)'=r^{N-1}\hat f(u)$ satisfies the following Pohozaev-type identity
\begin{multline}\label{eq:Pohozaev}
\left[r^N\left(\frac{|u'|^p}{p'}+\hat F(u)\right)+a r^{N-1}u \varphi_p(u') \right]' \\
=r^{N-1}\left[ \left(1-\frac{N}{p}+a\right)|u'|^p+N\hat F(u)-a \hat f(u)u \right],
\end{multline}
for every $a\in\R$.
\end{lemma}
\begin{proof}
On the one hand, by multiplying the equation by $ru'$, we have
\[
(r^{N-1}\varphi_p(u'))'ru'+r^{N}\hat f(u)u'=0.
\]
On the other hand, using the first relation in \eqref{eq:useful_rel}, we compute
\begin{multline*}
\left[r^N\frac{|u'|^p}{p'}\right]'
=\frac{N}{p'}r^{N-1}|u'|^p+\frac{r^N}{p'}\left[|\varphi_p(u')|^{p'}\right]' 
=\frac{N}{p'}r^{N-1}|u'|^p + r^N (\varphi_p(u'))'u'\\
=(r^{N-1}\varphi_p(u'))'ru' + \left(1-\frac{N}{p}\right) r^{N-1} |u'|^p.
\end{multline*}
Moreover, using the second relation in \eqref{eq:useful_rel}, we have
\[
[ar^{N-1}u \varphi_p(u')]'=ar^{N-1} |u'|^p-ar^{N-1}u \hat f(u).
\]
By combining the previous equalities, one obtains the statement.
\end{proof}

\begin{lemma}\label{lem:final}
Suppose that $\Omega=\mathcal B(R_2)$ and that $f$ satisfies $(f_\mathrm{subc})$. There exist three constants $K_1>0$, $K_2\geq0$ and $\bar d>1/\eta$ such that, for every $d\geq \bar d$, the solution $(u_d,v_d)$ of \eqref{eq:shooting} satisfies
\begin{equation}\label{eq:estimate_below1}
R_2^N H_d(r)
+\frac{N}{p^*} R_2^{N-1}\left( \frac{|u_d(r)|^p}{p}+\frac{|u_d'(r)|^p}{p'} \right) 
\geq K_1 f^*(d) d \left(r_0(d)\right)^N-K_2
\end{equation}
for every $r_0(d)\leq r\leq R_2$. Here $r_0(d)$ is the quantity defined in \eqref{eq:r0def} and we use the convention that $N/p^*=0$ in the case $p^*=+\infty$.
\end{lemma}
\begin{proof}
We consider the Pohozaev-type identity \eqref{eq:Pohozaev} with $a=N/p^*$ and integrate it in $[0,r]$, with $r_0(d)\leq r\leq R_2$. The Young's inequality provides
\begin{multline}\label{eq:estimate_below2}
R_2^N H_d(r)
+\frac{N}{p^*} R_2^{N-1}\left( \frac{|u_d(r)|^p}{p}+\frac{|u_d'(r)|^p}{p'} \right) \\
\geq \int_0^r \left[ N\hat F(u_d(t))-\frac{N}{p^*} \hat f(u_d(t))u_d(t) \right] t^{N-1} \,dt.
\end{multline}
In order to estimate the right hand side of \eqref{eq:estimate_below2}, we notice that assumption $(f_\mathrm{subc})$ implies the existence of $\bar s>1/\eta$ and $\delta>0$ with the property that
\begin{equation}\label{eq:NF}
N F(\eta s)-\frac{N}{p^*} f^*(s)s \geq \delta f^*(s)s >0,
\qquad \text{for every }s\geq \bar s.
\end{equation}
In particular, it also holds
\begin{equation}\label{eq:NF-afs}
N\hat F(s)>N F(\eta s)>\frac{N}{p^*} f^*(s)s\geq \frac{N}{p^*}\hat f(s)s,
\qquad \text{for every }s\geq \bar s.
\end{equation}
We split the right hand side of \eqref{eq:estimate_below2} into two parts that we estimate separately. Concerning the integral in $(r_0(d),r)$, we use relation \eqref{eq:NF-afs} to obtain
\begin{multline}\label{eq:estimate_below3}
\int_{r_0(d)}^r \left[ N\hat F(u_d(t))-\frac{N}{p^*} \hat f(u_d(t))u_d(t) \right] t^{N-1} \,dt \\
\geq \int_{\{t\in[r_0(d),r]:\, u_d(t)<\bar s\}} \left[ N\hat F(u_d(t))-\frac{N}{p^*} \hat f(u_d(t))u_d(t) \right] t^{N-1} \,dt \geq -C,
\end{multline}
where $C\geq0$ is a constant not depending on $d$ nor on $r$, and we have used the fact that $N\hat F(s)-\frac{N}{p^*} \hat f(s)s$ is bounded from below for $s\le \bar s$. 
As for the integral in $(0,r_0(d))$, we use Lemma \ref{lem:r0estimate}-(ii) and (iii) and relation \eqref{eq:NF} to get
\begin{multline}\label{eq:estimate_below4}
\int_0^{r_0(d)} \left[ N\hat F(u_d(t))-\frac{N}{p^*} \hat f(u_d(t))u_d(t) \right]t^{N-1} \,dt \\
\geq \left( N F(\eta d)-\frac{N}{p^*} f^*(d)d \right)\frac{(r_0(d))^N}{N} 
\geq \delta f^*(d)d\frac{(r_0(d))^N}{N},
\end{multline}
for every $d > \bar s$.

By combining \eqref{eq:estimate_below2}, \eqref{eq:estimate_below3} and \eqref{eq:estimate_below4} we obtain
\[
R_2^N H_d(r)
+\frac{N}{p^*} R_2^{N-1}\left( \frac{|u_d(r)|^p}{p}+\frac{|u_d'(r)|^p}{p'} \right) 
\geq \delta f^*(d)d\frac{(r_0(d))^N}{N} - C,
\]
for every $d > \bar s$ and $r_0(d)\leq r\leq R_2$.
Hence we have proved that the statement holds true with $\bar d = \bar s$, $K_1=\delta /N$ and $K_2=C$. 
\end{proof}

Using Lemmas \ref{lem:r0estimate}, \ref{lem:Pohozaev} and \ref{lem:final} we can finally give the proof of \eqref{eq:explosion_d} in the general subcritical case.

\begin{proposition}\label{prop:explosion3} Under the same assumptions as in Lemma \ref{lem:final}, 
property \eqref{eq:explosion_d} holds.
\end{proposition}
\begin{proof}
Again we can assume $F_\infty=+\infty$, since the complementary case was treated in Proposition \ref{prop:explosion1}-(i).
We aim to estimate from below the right hand side of \eqref{eq:estimate_below1}.

By \eqref{eq:r0estimate} and the fact that $f^*(s)\ge f(s) >0$ for $s>1$, we get for $d>1$
\begin{multline*}
K_1f^*(d)d(r_0(d))^N-K_2\ge K_1f^*(d)d\min\left\{R_2^N;[(1-\eta)dp']^{\frac{N}{p'}}\left(\frac{N}{f^*(d)}\right)^{\frac{N}{p}}\right\}-K_2\\
=K_1\min\left\{R_2^N f^*(d)d;[(1-\eta)p']^{\frac{N}{p'}}N^{\frac{N}{p}}d^{\frac{N}{p'}+1}f^*(d)^{1-\frac{N}{p}}\right\}-K_2.
\end{multline*}
We claim that both the terms in the above ``min'' go to infinity when $d \to +\infty$. Indeed, $f^*(d)d \to +\infty$ since $f^*$ is a positive non-decreasing function. As for the term $d^{\frac{N}{p'}+1}f^*(d)^{1-\frac{N}{p}}$, we distinguish two cases. If $N \leq p$, the conclusion is straightforward. In the case $N > p$, we first observe that, from relation \eqref{eq:subcrit} together with the fact that $f^*$ is the smallest non-decreasing function above $f$, it follows that
$$
f^*(s) \le  C_\varepsilon s^{p^*-1-\varepsilon} \quad \text{ for every } s>s_\varepsilon,
$$
Therefore, as $d \to +\infty$
$$
d^{\frac{N}{p'}+1}f^*(d)^{1-\frac{N}{p}}\ge (C_\varepsilon)^{1-\frac{N}{p}} d^{\frac{N}{p'}+1-\left(\frac{N}{p}-1\right)\left(p^*-1-\varepsilon \right)}=(C_\varepsilon)^{1-\frac{N}{p}}d^{\varepsilon\left(\frac{N}{p}-1\right)} \to +\infty.
$$

Thus we have obtained that the left hand side of \eqref{eq:estimate_below1} diverges as $d\to +\infty$, uniformly in $[r_0(d),R_2]$. So, in particular
\begin{equation}\label{Hr1}
\lim_{d\to+\infty}\left[R_2^N H_d(r_1(d))+\frac{N}{p^*} R_2^{N-1}\left( \frac{(u_d(r_1(d)))^p}{p}+\frac{|u_d'(r_1(d))|^p}{p'}\right) \right] =+\infty.
\end{equation}
We claim that $\lim_{d\to+\infty}H_d(r_1(d))=+\infty$. Indeed, suppose by contradiction that there exists a sequence $(d_n)_n$ such that $\lim_{n\to+\infty}d_n=+\infty$ and  $H_{d_n}(r_1(d_n))\le M $ for some $M\ge 0$ and for all $n$, then both $\frac1{p'}|u'_{d_n}(r_1(d_n))|^p\le M$ and $\hat F(u_{d_n}(r_1(d_n))\le M$ for all $n$. 
Since $F_\infty=+\infty$, the boundedness of $(\hat F(u_{d_n}(r_1(d_n)))_n$ implies that also $(u_{d_n}(r_1(d_n)))_n$ is bounded. This contradicts \eqref{Hr1} and proves the claim.
Finally, by  Remark~\ref{equivalence}, the conclusion follows. 
\end{proof}

\section{Proof of the main results}\label{Sec:main}

In this section, we take advantage of the elastic-type property \eqref{eq:explosion_d} to prove our main results, Theorems \ref{th:bounds}, \ref{th:C1>0} and \ref{th:C1=0}. 
We first show that \eqref{eq:explosion_d} has an immediate consequence on the sign of $u'_d$, for $d$ sufficiently large.

\begin{proposition}\label{Pr:d*}
Let $\Omega$ be either the annulus $\mathcal A(R_1,R_2)$ or the ball $\mathcal B(R_2)$ and let $f$ satisfy $(f_{\mathrm{reg}})$, $(f_{\mathrm{eq}})$, and $(f_0)$. 
In the case $\Omega=\mathcal B(R_2)$ assume in addition either $(f_\mathrm{subl})$ or $(f_\mathrm{subc})$.
There exists $d^* > 1$ such that if $(u_d,v_d)$ solves \eqref{eq:shooting}, then 
\begin{equation}\label{u-dec}
u'_d(r) < 0 \quad \mbox{ for every } r \in (R_1,R_2]\mbox{ and }d\ge d^*.
\end{equation}
In particular, if $u$ solves \eqref{main}, then $u(R_1)<d^*$.
\end{proposition}
\begin{proof}
Assume by contradiction that there exist a sequence $d_n$, with $d_n > 1$ and $d_n \to + \infty$, and a sequence $r_n \in (R_1,R_2]$ such that $u'_{d_n}(r_n) \geq 0$. Since $u_{d_n}(R_1)=d_n > 1$, by $(f_\mathrm{eq})$ we obtain that $u'_{d_n}(r)< 0$ in a right neighborhood of $R_1$ (compare with the proof of Lemma \ref{lem:r0estimate}-(i)); hence, we can assume without loss of generality that $r_n$ is the first minimum point of $u_{d_n}$. As a consequence, $u'_{d_n}(r_n) = 0$ and, using $(f_\textrm{eq})$ again, $0 < u_{d_n}(r_n) < 1$. Therefore, being $r_n$ the first minimum point for $u_{d_n}$, $r_n\le r_1(d_n)$. A contradiction with 
\eqref{eq:explosion_d} is therefore obtained. This implies that if $d\ge d^*$, $u'_d(R_2)\neq 0$ and consequently $u=u_d$ does not solve \eqref{main}.
\end{proof}

Using Proposition \ref{Pr:d*}, the proof of Theorem \ref{th:bounds} easily follows.

\begin{proof}[Proof of Theorem \ref{th:bounds}]
By Proposition \ref{Pr:d*}, if $u$ is a radial solution of \eqref{main} and, hence, $u_d=u$ is a solution of \eqref{ode2} for some $d > 0$, it has to be $d < d^*$. Using the very same arguments of the proof of Proposition~\ref{prop:explosion1}-(i)
we obtain, for every $d \in (0,d^*)$ and every $r \in [R_1,R_2]$,
\[
\frac{|u_d'(r)|^p}{p'} \leq \max\{\hat F(0),\hat F(d^*)\}
\]
and
\[  
\vert u_d(r) - d \vert \leq (p'\max\{\hat F(0),\hat F(d^*)\})^{1/p}(R_2-R_1).
\]
Hence the conclusion follows for $C = d^* + (p'\max\{\hat F(0),\hat F(d^*)\})^{1/p}(1+R_2-R_1)$.
\end{proof}

In the rest of the section we will prove Theorems \ref{th:C1>0} and \ref{th:C1=0}. 
As already mentioned in the Introduction, we prove here only the existence of solutions satisfying $u(R_1)>1$, and we refer to \cite[Theorem 1.2 and 1.4]{ABF} for the existence of solutions with $u(R_1)<1$. 
Therefore, having in mind the notation and the strategy described in Section \ref{Sec:shooting}, we can suppose $d>1$; from now on, $d^* > 1$ is the constant given by Proposition \ref{Pr:d*}. 

\begin{proof}[Proof of Theorem \ref{th:C1>0}] For every $d>1$, let $(\rho_d,\theta_d)$ be the solution of \eqref{sys2} with initial conditions \eqref{ci2}; notice that $\theta_d(R_1)=0$ because $d>1$.
On the one hand, the elastic-type property (see in particular relation \eqref{u-dec}) provides
$$
\theta_{d^*}(R_2)<\pi_p.
$$
On the other hand, the assumption $C_1>\lambda_{k+1}^{\mathrm{rad}}$ for some integer $k\geq1$, together with Lemma \ref{lem:d_close1}, provides the existence of $\bar\delta$ (which depends on $k$), such that
\[
\theta_d(R_2)> k\pi_p \quad\text{for all }d\in (1,1+\bar\delta].
\]

The previous two inequalities, together with the continuity of the map $d \mapsto \theta_d(R_2)$ (see Corollary \ref{cor:limit_rho_as_d_to0}), imply that for all $j=1,\dots,k$ 
there exists $d_j\in(1,d^*)$ for which $\theta_{d_j}(R_2)=j\pi_p$. This corresponds to $u_{d_j}'(R_2)=0$, providing the desired solutions $u_1,\ldots,u_k$ of \eqref{main}.

In order to prove the oscillatory behavior of each $u_j$ it suffices to remark that, since $\theta_{d_j}(r)$ is monotone increasing (see \eqref{sys2} and recall $(f_\textrm{eq})$), there exist exactly $j$ radii $r_1,\ldots,r_{j}\in (R_1,R_2)$ such that $\theta_{d_j}(r_1)=\frac{1}{2}\pi_p$, $\theta_{d_j}(r_2)=\frac{3}{2}\pi_p,\ldots,\theta_{d_j}(r_{j})=\left(j-\frac{1}{2}\right)\pi_p$.
\end{proof}

For the proof of Theorem \ref{th:C1=0} we still need a further result, which can be proved by combining the arguments used in the proof of \cite[Theorem~1.2]{ABF} (when $d \in (0,1)$) with Proposition \ref{Pr:d*}. Since the complete proof is quite long and it is not easy to summarize the required changes, for the reader's convenience we give all the details in a final Appendix. 

\begin{proposition}\label{prop:d_k}
Under the assumptions of Proposition \ref{Pr:d*}, for every $d>1$, let $(\rho_d,\theta_d)$ be the solution of \eqref{sys2} with initial conditions \eqref{ci2}.
For any integer $k \geq 1$ and any $\varepsilon > 0$,
there exists $R_*(k,\varepsilon) > 0$ such that if $R_1 < \varepsilon R_2$ and $R_2 > R_*(k,\varepsilon)$ there exists 
$\hat{d}_k \in (1,d^*)$ such that 
\begin{equation}\label{eq:thetaR2>}
\theta_{\hat{d}_k}(R_2) > k\pi_p.
\end{equation}
\end{proposition}

\begin{proof}[Proof of Theorem \ref{th:C1=0}] 
In view of Proposition \ref{prop:d_k}, it holds that $\theta_{\hat{d}_k}(R_2) > k\pi_p$ for some $\hat{d}_k > 1$. On the other hand, 
$\theta_d(R_2) < \pi_p$ both when $d$ is close enough to $1$ (by Lemma \ref{lem:d_close1} and the assumption $C_1 = 0 < 
\lambda_2^{\textnormal{rad}}$) and for $d = d^*$ (by Proposition~\ref{Pr:d*}).  

Then, by continuity, it is possible to find $2k$ values $d_j^\pm$ for $j=1\ldots,k$, with
\[
1 < d_1^-<\ldots<d_k^- < \hat{d}_k < d_k^+<\ldots<d_1^+ < d^* 
\]
and such that
$\theta_{d_j^\pm}(R_2) = j\pi_p$, $j=1\ldots,k$, giving rise to the desired solutions $u^\pm_j$. The oscillatory behavior is then proved as in Theorem \ref{th:C1>0}. In fact, by \eqref{sys2} and $(f_\textrm{eq})$, $\theta_{d_j^\pm}$ is increasing for every $j=1,\dots, k$, and consequently, there exist exactly $2j$ radii  $r_1^-,\ldots,r^-_{j}, r_1^+,\ldots,r^+_{j}\in (R_1,R_2)$ such that $\theta_{d^\pm_j}(r^\pm_1)=\frac{1}{2}\pi_p$, $\theta_{d^\pm_j}(r^\pm_2)=\frac{3}{2}\pi_p$,$\dots,$ $\theta_{d^\pm_j}(r^\pm_{j})=\left(j-\frac{1}{2}\right)\pi_p$.
\end{proof}

We close this section with two final remarks, discussing possible variants of our main results.

\begin{remark}\label{rem:C1=0}
We observe that the following statement, which can be considered as an intermediate result between Theorem \ref{th:C1>0} and Theorem \ref{th:C1=0}, holds true: 
\smallbreak
\noindent
\textit{Under the assumptions of Theorem \ref{th:bounds}, for any integer $k \geq 1$ and any $\varepsilon > 0$ there exists $R_*(k,\varepsilon) > 0$ such that if 
$$
R_1<\varepsilon R_2\quad\mbox{and}\quad R_2 > R_*(k,\varepsilon),
$$
then problem \eqref{main} has 
\begin{itemize}
\item[(i)] at least $2k$ distinct non-constant radial solutions;
\item[(ii)] at least $4k - 2(l-1)$ distinct non-constant radial solutions, if we further suppose that $(f_1)$ is satisfied with $C_1<\lambda_{l+1}^\mathrm{rad}$ for some $l=1,\ldots,k$.
\end{itemize}}
\smallbreak

The proof is really the same as the one of Theorem \ref{th:C1=0}. 
As for (i), Proposition~\ref{prop:d_k} yields the existence of $\hat{d}_k > 1$ such that $\theta_{\hat{d}_k}(R_2) > k \pi_p$ (notice indeed that the assumption $C_1 = 0$ is not used in the corresponding proof), so that $k$ radial solutions (such that $u(r) - 1$ has respectively $1,2,\ldots,k$ zeros) are found since $\theta_{d^*}(R_2) < \pi_p$. A symmetric argument works for $d \in (0,1)$, thus providing the $2k$ solutions mentioned in (i).

\begin{figure}[h!t]
\begin{center}
\includegraphics[scale=.8]{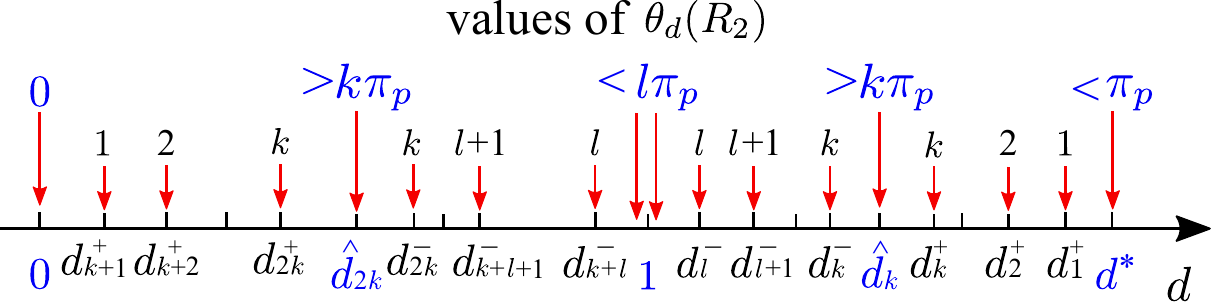}
\end{center}
\caption{Qualitative representation of the initial data $d$ to which correspond radial solutions of \eqref{main} in the case $C_1<\lambda_{l+1}^\mathrm{rad}$ for some $l=1,\ldots,k$.}
\label{fig:C1minore_l}
\end{figure}

Concerning (ii), the assumption $C_1<\lambda_{l+1}^\mathrm{rad}$ is used to ensure, by Lemma \ref{lem:d_close1}, that $\theta_d(R_2) < l \pi_p$ for $d$ close enough to $1$; as a consequence, $k - (l-1)$ further solutions (such that $u(r) - 1$ has respectively $l,l+1,\ldots,k$ zeros) appear. Since the same argument works for $d \in (0,1)$ as before, the conclusion follows, cf. Figure \ref{fig:C1minore_l}.

The drawback of this result is that (focusing for simplicity on the case of the ball $\Omega = \mathcal{B}(R_2)$) the conditions $C_1 < \lambda_{l+1}^\mathrm{rad}$ and $R_2 > R_*(k)$ are in competition with each other (since $\lambda_{l+1}^\mathrm{rad}(\mathcal{B}(R_2)) \to 0^+$ for $R_2 \to +\infty$) unless $C_1 = 0$, that is, in the case of Theorem \ref{th:C1=0}. For this reason we do not insist further on this topic; however, we think it is worth mentioning in order to better highlight the multiplicity scheme. \hfill $\diamond$
\end{remark}

\begin{remark}\label{rem:+zeri}
A careful inspection of the proofs shows that our multiplicity results are still valid when $f$ is defined on a compact interval $[0,\bar{u}]$, for some $\bar{u} > 1$, and satisfies:
\begin{itemize}
\item[$(f_\textrm{reg}')$] $f \in \mathcal{C}([0,\bar{u}]) \cap \mathcal{C}^1((0,\bar{u}))$;
\item[$(f_\textrm{eq}')$] $f(0) = f(1) = f(\bar{u})= 0$, $f(s) < 0$ for $0 < s < 1$ and $f(s) > 0$ for $1 < s < \bar{u}$; 
\item[$(f_0)$] $\liminf_{s\to0^+} \frac{f(s)}{s^{p-1}}  >-\infty$;
\item[$(f_1)$] there exists $C_1 := \lim_{s\to1} \frac{f(s)}{\vert s - 1 \vert^{p-2} (s-1)}$;
\item[$(f_{\bar{u}})$] $\liminf_{s\to \bar{u}^-} \frac{f(s)}{- (\bar{u} - s)^{p-1}} > - \infty$.
\end{itemize}
This situation can be interpreted as a limit-case of $(f_\mathrm{subl})$ (since one could extend $f$ outside $[0,\bar{u}]$ by setting $f(s) = 0$ for $s > \bar{u}$; notice however that $(f_\textrm{eq})$ would not be satisfied) but the proof is even simpler. Indeed, the function $u \equiv \bar{u}$ is now a further constant solution of \eqref{main}, so that $\theta_{\bar{u}} \equiv 0$: therefore, non-constant solutions with $d \in (1,\bar{u})$ can be proved to exist with the very same arguments used in \cite{ABF} for solutions with $d \in (0,1)$ (notice in particular that the machinery of Section \ref{Sec:elastic} of the present paper is not necessary). The only delicate point is that one has to assume the extra-condition $(f_{\bar{u}})$, which is needed to ensure the uniqueness of the Cauchy problem near $u = \bar{u}$ (in the same way as $(f_0)$ was needed in \cite[Lemma~2.2]{ABF} for the uniqueness near $u = 0$). We think that both assumption $(f_{\bar{u}})$ and $(f_0)$ could be removed via an approximation procedure on $f$ (hence giving rise to slightly generalized versions of the result results of this paper, as well as of the ones in \cite{ABF}) but we have preferred to focus on our simpler setting, avoiding further technicalities. \hfill $\diamond$
\end{remark}

\section{Appendix}\label{Sec:appendix}

We give below the proof of Proposition \ref{prop:d_k}.
We treat the two cases $\Omega=\mathcal B(R_2)$ and $\Omega=\mathcal A(R_1,R_2)$ simultaneously, by taking into account that the condition $R_1<\varepsilon R_2$ is trivially verified for all $\varepsilon>0$ when $R_1=0$, that is in the case of the ball. Hence, if $\Omega=\mathcal B(R_2)$, for any $k\ge1$ we can fix any $\varepsilon>0$ and consider $R_*$ only depending on $k$. 

It is convenient to write the equation in \eqref{ode2} as follows
\begin{equation}\label{sys3}
u' = \varphi_p^{-1}\left( \left(\frac{R_2}{r}\right)^{N-1} w\right), \qquad
w' = - \left(\frac{r}{R_2}\right)^{N-1} \hat f(u),
\end{equation}
for $r\in (R_1,R_2)$.
The advantage of this new scaling is that the maximum of $(r/R_2)^{N-1}$ in $[R_1,R_2]$ is independent of $R_2$ and that, at the same time, its minimum is positive in $[\varepsilon R_2,R_2]$ for any $\varepsilon>0$.
Comparing \eqref{sys3} with the first two equations in \eqref{eq:shooting}, it is immediately realized that, since
\begin{equation}\label{eq:relation_w_v}
w(r)=R_2^{1-N}v(r)
\end{equation}
all the properties discussed in Section \ref{Sec:shooting} still hold true for this slightly different planar formulation of \eqref{ode2}. In particular, we define  $(u_d,w_d)$ as the solution of \eqref{sys3} satisfying $(u_d(R_1),w_d(R_1)) = (d,0)$ and we pass to polar-like coordinates around the point $(1,0)$ as in \eqref{polari}, that is,
\begin{equation}\label{eq:p-polar_ell_phi}
\begin{cases}
u(r)-1=\ell(r)^\frac{2}{p} \cos_p(\phi(r))\\
w(r)=-\ell(r)^\frac{2}{p'} \sin_p(\phi(r)).
\end{cases}
\end{equation}
We thus obtain (compare with \eqref{sys2}) the system
\begin{equation}\label{sys4a}
\left\{
\begin{array}{l}
\displaystyle \ell'= \frac{p}{2\ell}\, \left(\frac{R_2}{r} \right)^{(N-1)(p'-1)} \varphi_{p'}(w) \, \left[\varphi_p(u-1)-\left(\frac{r}{R_2}\right)^{(N-1)p'} \hat f(u) \right] =: \mathrm{P}(r,\ell,\phi)
  \\
\displaystyle \phi'= \frac{1}{\ell^2} \left(\frac{R_2}{r} \right)^{(N-1)(p'-1)}\left[ (p-1) \vert w \vert^{p'}+\left(\frac{r}{R_2}\right)^{(N-1)p'} \hat f(u)(u-1) \right] =: \Theta(r,\ell,\phi),
\end{array}
\right.
\end{equation}
with initial conditions 
\begin{equation}\label{eq:ic_uw}
\ell(R_1)=(d-1)^{\frac{p}{2}},\quad \phi(R_1)=0.
\end{equation}
We denote $(\ell_d,\phi_d)$ the solution of \eqref{sys4a} and \eqref{eq:ic_uw}, and we remark that, by \eqref{eq:relation_w_v}, 
\begin{equation}\label{eq:theta_phi}
\theta_{d}(R_2) > k\pi_p \qquad\Leftrightarrow\qquad \phi_{d}(R_2) > k\pi_p.
\end{equation}
Furthermore, by \eqref{eq:p-polar_ell_phi} and recalling that $|\cos_p(\phi)|^p+(p-1)|\sin_p(\phi)|^{p'}=1$, it easily follows that
\begin{equation}\label{eq:ellequadro}
\ell^2=|u-1|^p+(p-1)|w|^{p'}.
\end{equation}
We observe for future use that another consequence of the elastic-type property \eqref{eq:explosion_d} is the following.

\begin{lemma}
Under the assumptions and notations above, there exists $\tilde d>1$ for which 
\begin{equation}\label{eq:tilde_d}
\ell_{\tilde d}(\varepsilon R_2)> 1.
\end{equation} 
\end{lemma}
\begin{proof}
For every $d>1$ we distinguish two cases. If $u_{d}(\varepsilon R_2)\le 0$, then $|u_{d}(\varepsilon R_2)-1|\ge 1$. Consequently, by \eqref{eq:ellequadro} and the fact that $u_d(\varepsilon R_2)$ and $w_d(\varepsilon R_2)$ cannot vanish simultaneously by the uniqueness of the solution,
$$\ell_d^2(\varepsilon R_2)=|u_d(\varepsilon R_2)-1|^p+(p-1)|w_d(\varepsilon R_2)|^{p'}>1.$$
If $u_{d}(\varepsilon R_2) > 0$, then $u_d(r)>0$ for every $r\in [R_1,\varepsilon R_2]$ (since by the equation in \eqref{ode2} and the definition of $\hat f$, if $u(\bar r)=0$, $u(r)\le 0$ for all  $r\in [\bar r, R_2]$), and so $\varepsilon R_2 < r_1(d)$, by the definition of $r_1(d)$ in \eqref{eq:r1def}. Hence, by the elastic-type property \eqref{eq:explosion_d}, 
$u_d(\varepsilon R_2)+|u'_d(\varepsilon R_2)|\to +\infty$ as $d\to +\infty$. So, in correspondence to any constant $M$, there exists $\tilde d$ for which
$u_{\tilde d}(\varepsilon R_2)+|u'_{\tilde d}(\varepsilon R_2)|\ge M$. This implies that  
$$
\ell_{\tilde d}^2(\varepsilon R_2)=|u_{\tilde d}(\varepsilon R_2)-1|^p+(p-1)\varepsilon^{(N-1)p'}|u'_{\tilde d}(\varepsilon R_2)|^{p}>1,
$$
where the last inequality holds for $M$ large enough.
\end{proof}

To proceed, we shall adapt an argument introduced in \cite{BosZan13} (see also \cite[Section~2]{BG} and \cite[Section~2.4]{ABF}). The idea is the following. Since the solutions $(u,w)$ of problem \eqref{sys3} wind clockwise around the point $(1,0)$ in the $(u,w)$ phase plane  (being $\phi'>0$ by \eqref{sys4a}), we can define two spiral-like curves 
$$\phi\mapsto (\ell_\pm^{2/p}(\phi)\cos_p(\phi),-\ell_\pm^{2/p'}(\phi)\sin_p(\phi))$$ (see \eqref{eqm} below) which bound the solution from below and from above respectively, in the phase plane. The control of the spirals allows to prove that there exists $\hat d_k>1$ for which the solution $(u_{\hat d_k},w_{\hat d_k})$ shot from $u(0)=\hat d_k$ is contained in an annular-like portion (cf. $\mathcal A_k$ below) of the phase plane centered at $(1,0)$ for all $r\in[\varepsilon R_2, \bar r]$, for some $\bar r\le R_2$ (i.e., there exist $\check\ell_k$ and $\hat\ell_k$ as in \eqref{choice} below for which $\ell(r)\in[\check\ell_k,\hat\ell_k]$ for all $r\in[\varepsilon R_2, \bar r]$). This in turn implies that the solution $(u_{\hat d_k}, w_{\hat d_k})$ performs around $(1,0)$ more than $k$ half-turns (i.e., $\phi_{\hat d_k}(R_2)>k\pi_p$, or equivalently $\theta_{\hat d_k}(R_2)>k\pi_p$).
 
\begin{proof}[Proof of Proposition \ref{prop:d_k}] Fix an integer $k\ge1$ and a real number $\varepsilon>0$. 
With reference to \eqref{sys4a}, let
$$
\mathrm{P}(r,\ell,\phi) =: \ell S\left(r,\ell^{\tfrac{2}{p}}\cos_p(\phi),
	-\ell^{\tfrac{2}{p'}}\sin_p(\phi)\right)
$$
and
$$
\Theta(r,\ell,\phi) =: U\left(r,\ell^{\tfrac{2}{p}}\cos_p(\phi);
-\ell^{\tfrac{2}{p'}}\sin_p(\phi)\right),
$$
hence, by \eqref{eq:ellequadro}, we have
$$
S(r,u,w) = \frac{p}{2}\, \left(\frac{R_2}{r} \right)^{(N-1)(p'-1)} \varphi_{p'}(w) \cdot 
\frac{\varphi_p(u-1)-\left(\dfrac{r}{R_2}\right)^{(N-1)p'} \hat f(u) }{\vert u-1 \vert^p + (p-1) \vert w \vert^{p'}}
$$
and
$$
U(r,u,w) = \left(\frac{R_2}{r} \right)^{(N-1)(p'-1)} \cdot \frac{(p-1) \vert w \vert^{p'}+\left(\dfrac{r}{R_2}\right)^{(N-1)p'} \hat f(u)(u-1) }{\vert u-1 \vert^p + (p-1) \vert w \vert^{p'}}.
$$
We also write
$$
\mathcal{M}_-(u,w) := \begin{cases}
\displaystyle{\frac{p}{2}\cdot \frac{\varphi_{p'}(w)\left( \varphi_p(u-1) - \hat f(u)\right)}{(p-1)\vert w \vert^{p'} + \hat f(u)(u-1)}}  &\text{if } w(u-1) \geq 0, \vspace{0.2cm} \\
\displaystyle{\frac{p}{2}\cdot \frac{\varphi_{p'}(w)\left( \varphi_p(u-1) - \varepsilon^{(N-1)p'} \hat f(u)\right)}{(p-1)\vert w \vert^{p'} + \varepsilon^{(N-1)p'} \hat f(u)(u-1)}} &\text{if } w(u-1) \leq 0,
\end{cases} 
$$ 
and
$$
\mathcal{M}_+(u,w) := \begin{cases}
\displaystyle{\frac{p}{2}\cdot \frac{\varphi_{p'}(w)\left( \varphi_p(u-1) - \varepsilon^{(N-1)p'} \hat f(u)\right)}{(p-1)\vert w \vert^{p'} + \varepsilon^{(N-1)p'} \hat f(u)(u-1)}} &\text{if } w(u-1) \geq 0, \vspace{0.2cm}\\
\displaystyle{\frac{p}{2}\cdot \frac{\varphi_{p'}(w)\left( \varphi_p(u-1) - \hat f(u)\right)}{(p-1)\vert w \vert^{p'} + \hat f(u)(u-1)}} &\text{if } w(u-1) \leq 0.
\end{cases} 
$$ 
A straightforward calculation shows that 
\begin{equation}\label{msu}
\mathcal{M}_-(u,w) \leq \frac{S(r,u,w)}{U(r,u,w)} \leq \mathcal{M}_+(u,w)
\end{equation}
for all $r \in [\varepsilon R_2,R_2]$ and all $(u,w) \in \mathbb{R}^2 \setminus \{(1,0)\}$. 

Then, we define $\ell_\pm(\phi;\bar\phi,\bar\ell)$ as the solution of 
\begin{equation}\label{eqm}
\begin{cases}
\frac{d\ell}{d\phi} = \ell \mathcal{M}_\pm \left(\ell^{2/p}\cos_p(\phi),-\ell^{2/p'}\sin_p(\phi)\right)\\
\ell_\pm(\bar\phi;\bar\ell,\bar\phi) = \bar\ell
\end{cases}
\end{equation}
and we set, for any $\bar\ell > 0$,
$$
m_k(\bar\ell) := \inf_{\bar\phi \in [0,2\pi_p),\;\phi \in [\bar\phi, \bar\phi + k \pi_p]} \ell_-(\phi;\bar\ell,\bar\phi),
$$
$$
M_k(\bar\ell) := \sup_{\bar\phi \in [0,2\pi_p), \; \phi \in [\bar\phi, \bar\phi + k \pi_p]} \ell_+(\phi;\bar\ell,\bar\phi).
$$
Since $\ell\equiv 0$ is a solution of the equation in \eqref{eqm}, by continuous dependence,  
$$\lim_{\bar\ell\to 0^+}M_k(\bar\ell)=0.$$ 
Moreover, by uniqueness $m_k(\bar\ell)>0$ for every $\bar\ell>0$. Hence, we can choose $0 < \check{\ell}_k < \ell^*_k < \hat \ell_k$ such that
\begin{equation}\label{choice}
0 < \check{\ell}_k < m_k(\ell^*_k) \leq \ell_k^* \leq M_k(\ell_k^*) < \hat \ell_k  < 1.
\end{equation}
Now, since $\ell_1(r) = 0$ for every $r\in[R_1,R_2]$ and $\ell_{\tilde d}(\varepsilon R_2) > 1$ (see \eqref{eq:tilde_d}), by continuity (see Corollary \ref{cor:limit_rho_as_d_to0})
\begin{equation}\label{eq:ell_d_k}
\mbox{there exists } \hat d_k \in (1,\tilde d) \mbox{ such that } \ell_{\hat d_k}(\varepsilon R_2) = \ell_k^*.
\end{equation}

Finally, we define
$$
\delta_k^* := \inf_{(u,w)\in \mathcal A_k} 
\frac{\varepsilon^{(N-1)p'} \hat f(u)(u-1) + (p-1)\vert w \vert^{p'}}{\vert u-1 \vert^p + (p-1)\vert w \vert^{p'}},
$$
where $\mathcal A_k:=\{(u,w)\in\mathbb R^2\,:\:\check{\ell}_k \leq \sqrt{|u-1|^p+(p-1)|w|^{p'}} \leq \hat \ell_k\}$.
We are now in a position to prove that, if $R_1 < \varepsilon R_2$ and
\begin{equation}\label{eq:R*def}
R_2 > R_*(k,\varepsilon) := \frac{k\pi_p}{(1-\varepsilon)\delta_k^*},
\end{equation}
then
\begin{equation}\label{finalclaim}
\phi_{\hat d_k}(R_2) - \phi_{\hat d_k}(\varepsilon R_2) > k \pi_p
\end{equation}
and so $\phi_{\hat d_k}(R_2) > k\pi_p$, being $\phi_{\hat d_k}(\varepsilon R_2)>\phi_{\hat d_k}(R_1)= 0$ by the monotonicity of $\phi_d$. 
In particular, we have that $\hat d_k<d^*$, since by \eqref{u-dec} $\phi_{d}(R_2)<\pi_p$ for every $d \geq d^*$, and in turn, thanks to \eqref{eq:theta_phi}, that \eqref{eq:thetaR2>} holds, thus concluding the proof.

In order to prove \eqref{finalclaim}, we distinguish two cases. If $\ell_{\hat d_k}$ satisfies $\ell_{\hat d_k}(r) \in [\check{\ell}_k,\hat \ell_k]$ for any $r \in [\varepsilon R_2, R_2]$, i.e. $(u_{\hat d_k}(r),w_{\hat d_k}(r)) \in \mathcal{A}_k$ for every $r \in [\varepsilon R_2, R_2]$,
we easily conclude. Indeed, by the expression of $\phi'$ in \eqref{sys4a}, the definition of $\delta_k^*$ and the hypothesis on $R_2$,
\begin{multline*}
\phi_{\hat d_k}(R_2) - \phi_{\hat d_k}(\varepsilon R_2) 
= \int_{\varepsilon R_2}^{R_2} \phi_{\hat d_k}'(r)\,dr \\
\geq \int_{\varepsilon R_2}^{R_2} \frac{\varepsilon^{(N-1)p'} \hat f(u_{\hat d_k})(u_{\hat d_k}-1) + (p-1)\vert w_{\hat d_k} \vert^{p'}}{\vert u_{\hat d_k}-1 \vert^p + (p-1)\vert w_{\hat d_k} \vert^{p'}} \,dr 
\geq R_2 (1-\varepsilon) 
\delta_k^* > k\pi_p.
\end{multline*}
Otherwise, we let $\bar r \in [\varepsilon R_2, R_2)$ be the largest value such that $\ell_{\hat d_k}(r) \in [\check{\ell}_k,\hat \ell_k]$ for any
$r \in [\varepsilon R_2, \bar{r}]$. 
Such $\bar r$ exists because, by \eqref{eq:ell_d_k} and \eqref{choice}, $\ell_{\hat d_k}(\varepsilon R_2)\in [\check \ell_k,\hat \ell_k]$.
In this case we prove that
$$
\phi_{\hat d_k}(\bar r) - \phi_{\hat d_k}(\varepsilon R_2) > k\pi_p,
$$
implying \eqref{finalclaim} again in view of the monotonicity of $\phi_{\hat d_k}$.

Suppose by contradiction that this is not true and, just to fix the ideas, that $\ell_{\hat d_k}(\bar r) = \hat \ell_k$ (in the case $\ell_{\hat d_k}(\bar r) = \check \ell_k$ the argument is analogous). Observe also that, again by the monotonicity of 
$\phi_{\hat d_k}$, we have $\phi_{\hat d_k}(r) - \phi_{\hat d_k}(\varepsilon R_2) \leq k \pi_p$ for any $r \in [\varepsilon R_2,\bar r]$.
Now, we consider the function $\gamma(\phi) = \ell^+(\phi;\ell_k^*,\bar\phi)$, where $\bar\phi \in [0,2\pi_p)$ is such that
$\phi_{\hat d_k}(\varepsilon R_2) \equiv \bar\phi \mod 2\pi_p$. By the definition of $M_k(\ell_k^*)$ and \eqref{choice}, it holds
$$
\gamma(\phi) < \hat \ell_k\quad \mbox{ for every } \phi \in [\bar\phi,\bar\phi + k\pi_p];
$$
moreover, from \eqref{msu} and \eqref{eqm} we obtain non ricordo come ottenere la disuguaglianza qui sotto
$$
\mathrm{P}(r,\gamma(\phi),\phi) \leq \gamma'(\phi)\Theta(r,\gamma(\phi),\phi) \;\; \mbox{ for every } r \in 
[\varepsilon R_2,\bar{r}] \mbox{ and } \phi \in [\bar\phi,\bar\phi + k \pi_p].
$$
By \cite[Lemma 2.8]{ABF} (cf. \cite[Corollary 5.1]{BosZan13}), this implies that 
$$
\ell_{\hat d_k}(r) \leq \gamma(\phi_{\hat d_k}(r)) \quad \mbox{ for every } r \in [\varepsilon R_2,\bar{r}],
$$
so that $\ell_{\hat d_k}(\bar{r}) \leq M_k(\ell_k^*) < \hat \ell_k$, a contradiction.

In conclusion, we have proved that if $R_1<\varepsilon R_2$ and $R_2$ satisfies \eqref{eq:R*def}, then there exists $\hat d_k\in(1,d^*)$ such that \eqref{finalclaim} holds. As already noticed, this is enough to conclude since, by \eqref{eq:theta_phi}, this implies that \eqref{eq:thetaR2>} holds. 
\end{proof}

\section*{Acknowledgments}
\noindent A. Boscaggin and B. Noris acknowledge the support of the project ERC Advanced Grant 2013 n. 339958: ``Complex Patterns for Strongly Interacting Dynamical Systems --
COMPAT''.  F. Colasuonno acknowledges the supports of the Laboratoire Ami\'enois de Math\'ematique Fondamentale et Appliqu\'ee for her visit at Amiens, and of the ``National Group for Mathematical Analysis, Probability and their Applications'' (GNAMPA - INdAM) for her participation in the event ``Intensive week of PDEs at Spa'', where parts of this work have been achieved.
A. Boscaggin and F. Colasuonno were partially supported by the INdAM - GNAMPA Projects 2017 ``Dinamiche complesse per il problema degli
$N$-centri'' and ``Regolarit\`a delle soluzioni viscose per equazioni a derivate parziali non lineari degeneri", respectively. 

\bibliographystyle{abbrv}
\bibliography{biblio}

\end{document}